\documentclass[11pt]{amsart}
\usepackage[margin=1.05in]{geometry}
\usepackage{amsmath,amssymb,amsthm,mathtools}
\usepackage{booktabs,longtable,array}
\usepackage[hidelinks]{hyperref}
\usepackage{enumitem}
\usepackage{etoolbox}

\makeatletter
\patchcmd{\@setaddresses}{\addvspace\bigskipamount}{\addvspace\smallskipamount}{}{}
\patchcmd{\@setaddresses}{\addvspace\bigskipamount}{\addvspace\smallskipamount}{}{}
\makeatother

\newtheorem{theorem}{Theorem}[section]
\newtheorem{proposition}[theorem]{Proposition}
\newtheorem{lemma}[theorem]{Lemma}
\newtheorem{corollary}[theorem]{Corollary}
\theoremstyle{definition}
\newtheorem*{maintheorem}{Main Theorem}
\theoremstyle{remark}
\newtheorem{remark}[theorem]{Remark}
\newtheorem{question}[theorem]{Question}

\newcommand{\Z}{\mathbb Z}
\newcommand{\Q}{\mathbb Q}

\newcommand{\R}{\mathbb R}
\newcommand{\Sp}{\operatorname{Sp}}
\newcommand{\Tor}{\operatorname{Tor}}

\newcommand{\coker}{\operatorname{coker}}
\newcommand{\rk}{\operatorname{rk}}

\title[Higher mixed sums]{Higher Braided Mixed Sums, Intersection Lattices, and Symplectic Four-Manifolds}
\makeatletter
\def\@settitle{\begin{center}%
  \baselineskip14\p@\relax
  \bfseries
  \uppercasenonmath\@title
  \@title\par
  \end{center}}
\makeatother
\author{Anar Akhmedov}
\address{School of Mathematics, University of Minnesota, Minneapolis, MN 55455, USA}
\address{Department of Mathematics, Harvard University, Cambridge, MA 02138, USA}
\email{akhmedov@umn.edu}
\author{Azer Akhmedov}
\address{Department of Mathematics, North Dakota State University, Fargo, ND 58108, USA}
\email{azer.akhmedov@ndsu.edu}
\date{}

\begin{document}
\begin{abstract}
We study mixed symplectic sums along connected unbranched cyclic multisections of degrees $p,q\geq 2$ in ruled surfaces.  With
\[
p(g-1)=q(h-1)=n,
\]
the resulting four-manifold has signature zero and Euler characteristic
\[
e=4n\left(1-\frac1p-\frac1q\right).
\]
We classify the types having the rational cohomology ring of $\#_k(S^2\times S^2)$.  They occur only for odd $k$ and, with $t=(k+1)/2$, $a=g-1$, and $b=h-1$, are exactly the solutions of
\[
n=a+b+t,\qquad a\mid b+t,\qquad b\mid a+t.
\]
For each odd $k$ there are finitely many types; we give an explicit parametrization and tables for $k\leq 15$.  Every type admits an equatorially product-framed symplectic realization and, possibly with different framings, a spin realization whose free integral intersection form is $kH$.  We prove minimality, boundary incompressibility, an integral Mayer--Vietoris presentation, and a regular cyclic cover of degree $\operatorname{lcm}(p,q)$.  The boundary subgroup $\pi_1(\Sigma_G)\times\mathbb Z$ embeds in the fundamental group; this gives a nontrivial Bass--Serre splitting and rules out word-hyperbolicity, cocompact real- or complex-hyperbolic lattices, and torsion-free cocompact irreducible bidisk lattices.  There are five types for $k=3$ and seven for $k=5$.  The $k=1$ cases are treated in \cite{AkhmedovBraided,AkhYeung}.
\end{abstract}

\maketitle
\enlargethispage{3pt}

\section{Introduction}

The use of connected symplectic representatives of multiple-section classes in ruled surfaces goes back to Smith's work on braided symplectic submanifolds \cite{Smith}.  It was used by the first author in \cite{AkhNote}, in particular for a connected genus-three representative of the class $2[\Sigma_2\times\{\mathrm{pt}\}]$, and was developed into the explicit cyclic multisection construction used here in \cite{AkhmedovBraided}.

In \cite{AkhmedovBraided}, mixed fiber sums along explicit cyclic multisections were used to construct symplectic four-manifolds with the rational cohomology of $S^2\times S^2$.  If
\[
F_{g,p}\subset \Sigma_g\times S^2,\qquad F_{h,q}\subset \Sigma_h\times S^2
\]
are connected unbranched multisections of degrees $p$ and $q$, respectively, and have the same genus, then
\[
p(g-1)=q(h-1)=n
\]
and the common genus is $G=n+1$.  The mixed sum
\[
Z^{p,q}_{g,h}(\varphi)=
(\Sigma_g\times S^2\setminus \operatorname{int}\nu F_{g,p})
\cup_{\varphi}
(\Sigma_h\times S^2\setminus \operatorname{int}\nu F_{h,q})
\]
has
\[
e(Z)=4n\left(1-\frac1p-\frac1q\right),\qquad \sigma(Z)=0.
\]
The condition $e(Z)=4$ leads to the three possibilities
\[
(2,3;4,3;7),\qquad (2,4;3,2;5),\qquad (3,3;2,2;4),
\]
where we list $(p,q;g,h;G)$.  For suitable gluings, all three have the rational cohomology ring of $S^2\times S^2$.  The same paper asked for the higher mixed sums having the rational cohomology ring of
\[
\#_k(S^2\times S^2).
\]

The relation with the earlier constructions is as follows.  The paper
\cite{AkhNote} supplied an earlier use of connected multiple-section
classes in ruled surfaces in symplectic-sum constructions, while
\cite{AkhGokova} constructed minimal symplectic manifolds with the integral
cohomology of $S^2\times S^2$ by different surgery and twisted-sum methods.
The explicit root multisections, their complement topology, and the three
$k=1$ mixed sums were established in \cite{AkhmedovBraided}.  The new
points here are the complete higher numerical classification, the
rectangular Mayer--Vietoris realization for every admissible type, the
general spin and lattice theorem, and the higher cyclic-cover and
Bass--Serre analysis.  The spin and cyclic-cover viewpoint for the three
smallest cases is developed jointly in \cite{AkhYeung}; here it is extended
to every higher numerical type.

The higher problem posed in \cite{AkhmedovBraided} asks which mixed sums can have the rational cohomology ring of
\[
\#_k(S^2\times S^2),\qquad k\geq 3.
\]
The classification below is complete for every $k$, rather than only for a bounded range.  If
\[
t=\frac{k+1}{2},\qquad a=g-1,\qquad b=h-1,
\]
then the mixed-sum equation becomes
\[
n=a+b+t,
\]
and the covering degrees are integral precisely when
\[
a\mid b+t,\qquad b\mid a+t.
\]
Thus the geography problem becomes a finite divisibility problem, and no numerical type occurs for even $k$.

The realization theorem shows that every numerical type allowed by the divisibility conditions admits an equatorially product-framed symplectic gluing with vanishing rational first homology, and hence is represented by a symplectic four-manifold with the rational cohomology ring of $\#_k(S^2\times S^2)$.  For $k>1$ the Mayer--Vietoris surface map is rectangular, so the determinant condition of the smallest case is replaced by a maximal-rank condition.

The topology surviving the gluing is controlled by the complement
\[
M_{g,p}=(\Sigma_g\times S^2)\setminus\operatorname{int}\nu F_{g,p}
\]
which is a $p$-punctured-sphere bundle over $\Sigma_g$ and satisfies
\[
H_1(M_{g,p};\mathbb Z)\cong\mathbb Z^{2g}\oplus\mathbb Z/p.
\]
Section~5 gives the integral Mayer--Vietoris presentation for the higher sums.  The boundary is incompressible, so every direct sum contains
\[
\pi_1(\Sigma_G)\times\mathbb Z
\]
in its fundamental group.  In particular, changing the boundary gluing alone cannot make a direct sum simply connected.  This subgroup also gives the real-hyperbolic and irreducible bidisk-lattice obstructions discussed later.

The spin argument developed in \cite{AkhYeung} extends to every higher numerical type.  After choosing spin-compatible symplectic normal framings, the gluing can be chosen so that $H_1(Z;\mathbb Q)=0$ and $w_2(Z)=0$.  These framings need not be the equatorial framings used for the split integral formula in Section~5.  The free integral intersection lattice is then
\[
Q_{Z,\mathrm{free}}\cong kH.
\]
All the symplectic direct sums considered here are minimal.  We also construct a cyclic cover of degree $\operatorname{lcm}(p,q)$ after fixing the cyclic characters and their embeddings in the deck group, and determine the numerical restrictions on any orientation-compatible complex structure.

The main statements are collected in the following theorem.

\begin{maintheorem}
Let $p,q\geq 2$ and $g,h\geq 2$, and suppose
\[
p(g-1)=q(h-1)=n.
\]
A direct cyclic mixed sum can have the rational cohomology ring of $\#_k(S^2\times S^2)$ only when $k$ is odd.  For odd $k$, put $t=(k+1)/2$, $a=g-1$, and $b=h-1$.  Then the numerical types are exactly the solutions of
\[
n=a+b+t,\qquad a\mid b+t,\qquad b\mid a+t.
\]
For every such type there is a minimal equatorially product-framed symplectic realization $Z$ with
\[
H^*(Z;\mathbb Q)\cong H^*(\#_k(S^2\times S^2);\mathbb Q).
\]
After a possibly different choice of spin-compatible symplectic normal framings, the type also admits a spin realization, product-framed relative to those framings, for which $Q_{Z,\mathrm{free}}\cong kH$.  For every direct gluing,
\[
\pi_1(\Sigma_G)\times\mathbb Z\hookrightarrow\pi_1(Z).
\]
\end{maintheorem}

For $k=3$ there are five numerical types, and for $k=5$ there are seven.  We study the $k=3$ cases in detail, concentrating on the symmetric degree pairs $(4,4)$ and $(3,3)$ and the asymmetric pair $(6,2)$.  For each we describe the complements, the integral Mayer--Vietoris map, the spin lattice, and the associated cyclic cover.  The $(3,3)$ and $(6,2)$ examples have the same gluing genus but different complement homology and lifted decompositions; the gluing genus alone does not determine the higher type.

Boundary incompressibility also connects the construction with geometric group theory.  The direct sum is a nontrivial amalgamated product over $\pi_1(\Sigma_G)\times\mathbb Z$, and each vertex group has a finite-index subgroup $\pi_1(\Sigma_G)\times F_r$.  The associated Bass--Serre action has an edge group containing $\mathbb Z^2$, which excludes cocompact real- and complex-hyperbolic lattices; a separate centralizer argument excludes torsion-free cocompact irreducible bidisk lattices.  Section~6 compares this splitting with the real-hyperbolic hybrids of Gromov--Piatetski-Shapiro and with Deligne--Mostow lattices and complex-hyperbolic hybrid constructions.  The comparison concerns their graph-of-groups structure, not their lattice constructions.

The paper is organized as follows.  Section 2 recalls the cyclic multisections and the mixed sum.  Section 3 solves the higher numerical classification problem.  Sections 4 and 5 develop the rational and integral Mayer--Vietoris calculations.  Section 6 proves boundary incompressibility, describes the Bass--Serre splitting, and establishes the lattice obstructions; Section 7 proves minimality.  Section 8 gives the spin realization theorem and the free intersection lattice.  Sections 9 and 10 treat cyclic covers associated to the mixed-sum presentation and the twisted ruled analogue.  Section 11 studies the first two higher levels, with detailed treatment of three $k=3$ models.  Section 12 gives integral and complex-geometric consequences.  The final section records the remaining questions, and Appendix A contains the complete numerical list through $k=15$.

\section{Cyclic multisections and the mixed sum}

We recall the construction in the form needed below.  Let $g\geq 1$ and choose a primitive class $\xi\in H^1(\Sigma_g;\Z)$.  Represent it by a smooth map
\[
u:\Sigma_g\longrightarrow S^1.
\]
Identify $S^1$ with an equator in $S^2$.  For $p\geq 2$, set
\[
F_{g,p}=\{(x,z)\in \Sigma_g\times S^1:z^p=u(x)\}.
\]

\begin{proposition}\label{prop:multi}
The surface $F_{g,p}$ is connected, embedded, and symplectic.  The projection $F_{g,p}\to \Sigma_g$ is an unbranched cyclic cover of degree $p$, and
\[
[F_{g,p}]=p[\Sigma_g\times\{\mathrm{pt}\}],\qquad F_{g,p}^2=0,
\]
while
\[
g(F_{g,p})=p(g-1)+1.
\]
\end{proposition}

\begin{proof}
The equation $z^p=u(x)$ is the pullback of the standard degree-$p$ covering $S^1\to S^1$.  Since $z\mapsto z^p$ is a local diffeomorphism, the fiber product is a smooth embedded surface even over critical points of $u$.  Primitivity of $\xi$ makes the induced character onto $\Z/p$, hence the pullback cover is connected.  If
\[
[F_{g,p}]=a[\Sigma_g\times\{\mathrm{pt}\}]+b[\{\mathrm{pt}\}\times S^2],
\]
intersection with a sphere fiber gives $a=p$, while intersection with a section whose $S^2$-coordinate is disjoint from the equator gives $b=0$.  Thus $[F_{g,p}]=p[\Sigma_g\times\{\mathrm{pt}\}]$ and $F_{g,p}^2=0$.  The genus formula is Riemann--Hurwitz for an unbranched cover.  For a product symplectic form, the vertical component of the tangent plane lies in the one-dimensional tangent space of the equator.  The $S^2$-area term therefore restricts trivially, while the pullback of the area form on the base is positive.  Hence $F_{g,p}$ is symplectic.
\end{proof}

Suppose now that $F_{g,p}$ and $F_{h,q}$ have the same genus.  Then
\begin{equation}\label{eq:n}
p(g-1)=q(h-1)=n,
\end{equation}
and the common genus is
\[
G=n+1.
\]
Choose an orientation-preserving diffeomorphism of the two multisections and reverse the normal-circle orientation.  The resulting symplectic sum will be denoted
\[
Z^{p,q}_{g,h}(\varphi).
\]

\begin{proposition}\label{prop:geo}
The mixed sum satisfies
\[
e(Z^{p,q}_{g,h})=4n\left(1-\frac1p-\frac1q\right),\qquad
\sigma(Z^{p,q}_{g,h})=0.
\]
\end{proposition}

\begin{proof}
Since $g=1+n/p$, $h=1+n/q$, and $e(\Sigma_G)=2-2G=-2n$, the fiber-sum formula gives
\begin{align*}
e(Z)&=(4-4g)+(4-4h)-2(2-2G)\\
&=4n\left(1-\frac1p-\frac1q\right).
\end{align*}
Both product ruled surfaces have signature zero, so Novikov additivity gives $\sigma(Z)=0$.
\end{proof}

We fix the gluing terminology used below.  After choosing oriented normal
framings, write \(\partial M_{g,p}=F_{g,p}\times S^1\).  A
\emph{direct gluing} is a boundary
diffeomorphism which identifies the two multisections by an
orientation-preserving diffeomorphism and reverses the orientation of the
normal circle.  In homology it has the form
\[
 (v,r)\longmapsto (Sv,\lambda(v)-r),
\]
where \(S\in\Sp(2G,\mathbb Z)\) and
\(\lambda:H_1(\Sigma_G;\mathbb Z)\to\mathbb Z\) records a change of normal
framing relative to the specified trivializations.  A direct gluing is
\emph{product-framed} relative to those framings when \(\lambda=0\).
Every integral framing term is represented by a map to
$SO(2)=U(1)$ and hence by a symplectic normal gauge transformation.
After the areas of the two multisections have been matched, these boundary
maps may be used in the symplectic sum construction \cite{Gompf}: an
orientation-preserving surface diffeomorphism may be isotoped, after
matching the area forms, to an area-preserving one.  Unless a framing
change is explicitly mentioned, all realization results below use the
product framing.  Section~5 uses the adapted equatorial framings, whereas
Section~8 allows spin-compatible symplectic framings.

\section{The higher classification equation}

We begin with the numerical classification.

\begin{theorem}\label{thm:classification}
Assume $k\geq1$, $p,q\geq 2$, $g,h\geq 2$, and \eqref{eq:n}.  Suppose that a mixed sum has the rational cohomology ring of $\#_k(S^2\times S^2)$.  Then $k$ is odd.  Put
\[
t=\frac{k+1}{2},\qquad a=g-1,\qquad b=h-1.
\]
Then
\begin{equation}\label{eq:class}
n=a+b+t,
\end{equation}
and
\begin{equation}\label{eq:div}
a\mid b+t,\qquad b\mid a+t.
\end{equation}
Conversely, every pair of positive integers $(a,b)$ satisfying \eqref{eq:div}, with $n$ defined by \eqref{eq:class}, determines integers
\[
p=\frac na,\qquad q=\frac nb,
\]
with $p,q\geq 2$, and the resulting mixed sum has
\[
e=2+2k,\qquad \sigma=0.
\]
\end{theorem}

\begin{proof}
The connected sum $\#_k(S^2\times S^2)$ has Euler characteristic $2+2k$ and signature zero.  Proposition~\ref{prop:geo} therefore gives
\[
4n\left(1-\frac1p-\frac1q\right)=2+2k.
\]
Since
\[
n\left(1-\frac1p-\frac1q\right)=n-a-b\in\mathbb Z,
\]
the left side is divisible by four.  Thus $2+2k$ is divisible by four and $k$ is odd.  Writing $t=(k+1)/2$ gives
\[
n-\frac np-\frac nq=t.
\]
Since $a=n/p$ and $b=n/q$, this is exactly \eqref{eq:class}.  The conditions $p,q\in\Z$ are equivalent to $a\mid n$ and $b\mid n$.  Using $n=a+b+t$, these become \eqref{eq:div}.

Conversely, if \eqref{eq:div} holds, then $a\mid n$ and $b\mid n$, so $p=n/a$ and $q=n/b$ are integers.  Since $n>a,b$, both are at least two.  Substitution in Proposition~\ref{prop:geo} gives $e=4t=2+2k$.
\end{proof}

The divisibility form gives an explicit finite parametrization.

\begin{theorem}\label{thm:param}
Fix $t\geq 1$.  Up to interchanging the two summands, assume $a\leq b$ and write
\[
b=a+d,\qquad d\geq 0.
\]
Then $0\leq d\leq t$.  If $0\leq d<t$, all solutions are obtained by choosing $a>0$ such that
\begin{equation}\label{eq:param1}
a\mid t+d,\qquad a+d\mid t-d.
\end{equation}
If $d=t$, all solutions are obtained by choosing
\begin{equation}\label{eq:param2}
a\mid 2t,
\end{equation}
and setting $b=a+t$.
In particular, for each fixed odd $k$ there are only finitely many numerical mixed-sum types.
\end{theorem}

\begin{proof}
From \eqref{eq:div}, with $b=a+d$, we have
\[
a\mid d+t
\]
and
\[
a+d\mid a+t.
\]
Subtracting $a+d$ from the second dividend gives
\[
a+d\mid t-d.
\]
If $d>t$, then $0<d-t<a+d$, so this divisibility is impossible.  Thus $d\leq t$.

If $d<t$, the two divisibility conditions are exactly \eqref{eq:param1}.  If $d=t$, the second condition is automatic and the first becomes $a\mid 2t$.  These conditions are also sufficient by reversing the calculation.
\end{proof}

\begin{remark}\label{rem:degrees}
For $d<t$ one may write the corresponding degrees directly as
\[
p=2+\frac{t+d}{a},\qquad
q=2+\frac{t-d}{a+d}.
\]
For $d=t$ one has
\[
p=2+\frac{2t}{a},\qquad q=2.
\]
Thus, for $d<t$, the solutions are generated from the indicated divisor conditions involving $t-d$ and $t+d$; for $d=t$, they are generated by the positive divisors of $2t$.
\end{remark}

\subsection{The first numerical cases}

The case $k=1$, $t=1$, recovers the three numerical types of \cite{AkhmedovBraided}.  Since the first two higher levels are studied in detail below, we record here only the cases $k=1,3,5$.  The complete list through $k=15$ is given in Appendix~\ref{app:table}.  Each row records $(p,q;n;g,h;G)$.

\begin{center}
\small
\begin{tabular}{c c c c c c c}
\toprule
$k$ & $p$ & $q$ & $n$ & $g$ & $h$ & $G$\\
\midrule
1&3&3&3&2&2&4\\
1&4&2&4&2&3&5\\
1&3&2&6&3&4&7\\
\midrule
3&4&4&4&2&2&5\\
3&3&3&6&3&3&7\\
3&6&2&6&2&4&7\\
3&4&2&8&3&5&9\\
3&3&2&12&5&7&13\\
\midrule
5&5&5&5&2&2&6\\
5&3&3&9&4&4&10\\
5&6&3&6&2&3&7\\
5&8&2&8&2&5&9\\
5&5&2&10&3&6&11\\
5&4&2&12&4&7&13\\
5&3&2&18&7&10&19\\
\bottomrule
\end{tabular}
\end{center}

\begin{corollary}\label{cor:counts}
For $k=1,3,5,7,9,11,13,15$, the numbers of unordered numerical types are respectively
\[
3,5,7,7,9,11,10,10.
\]
\end{corollary}

\begin{proof}
Apply Theorem~\ref{thm:param} for $t=1,\ldots,8$.
\end{proof}

\section{Rational homology realization}

Put
\[
M_{g,p}=(\Sigma_g\times S^2)\setminus \operatorname{int}\nu F_{g,p}.
\]
The projection to the base makes $M_{g,p}$ a bundle over $\Sigma_g$ with fiber the $p$-punctured sphere $P_p$.  The cyclic monodromy permutes the boundary circles.  As in \cite{AkhmedovBraided},
\begin{equation}\label{eq:H1comp}
H_1(M_{g,p};\Z)\cong H_1(\Sigma_g;\Z)\oplus \Z/p.
\end{equation}
Here and below the finite summand is generated by a meridian of the
multisection.  We include the calculation because it is used in every
integral argument.  Choose a symplectic basis
\(a_1,b_1,\ldots,a_g,b_g\) so that the covering character sends \(a_1\)
to \(1\in\mathbb Z/p\) and all other basis elements to zero.  If
\(x_1,\ldots,x_p\) are the oriented boundary curves of \(P_p\), the only
nontrivial monodromy cyclically permutes the \(x_j\), and
\(x_1+\cdots+x_p=0\) in \(H_1(P_p;\mathbb Z)\).  In the coinvariants all
the \(x_j\) become a single class \(\mu_p\), and the last relation becomes
\(p\mu_p=0\).  The north-pole section of the complement bundle splits the
map to the base.  Equivalently, it splits the fundamental-group extension,
so abelianization is the direct sum of the base homology and the fiber
coinvariants; in particular the possible transgression vanishes.  Hence
\[
 H_1(M_{g,p};\mathbb Z)
 =\langle a_1,b_1,\ldots,a_g,b_g,\mu_p\mid p\mu_p=0\rangle,
\]
which proves \eqref{eq:H1comp} and identifies the meridional class.
Let
\[
Q_{p,g}:H_1(F_{g,p};\Z)\longrightarrow H_1(\Sigma_g;\Z)
\]
be the push-forward of the covering map.  We shall use the following elementary observation.

\begin{lemma}\label{lem:covermap}
For the degree-$p$ covering $q:F_{g,p}\to\Sigma_g$, the map
\[
Q_{p,g}=q_*:H_1(F_{g,p};\mathbb Q)\longrightarrow H_1(\Sigma_g;\mathbb Q)
\]
is surjective.  Its kernel is a symplectic subspace of dimension $2G-2g$ in $H_1(F_{g,p};\mathbb Q)$.
\end{lemma}

\begin{proof}
The transfer, or equivalently the pullback on homology after identifying homology and cohomology by Poincar\'e duality, gives a map $q^!$ satisfying
\[
q_*q^!=p\,\mathrm{id}.
\]
Hence $q_*$ is surjective over $\mathbb Q$.  Moreover $q_*$ and $q^!$ are adjoint for the algebraic intersection pairings.  Thus
\[
\ker q_*=(\operatorname{im}q^!)^\perp.
\]
The restriction of the intersection form to $\operatorname{im}q^!$ is $p$ times the intersection form downstairs and is therefore nondegenerate.  Its symplectic orthogonal complement is consequently nondegenerate as well.  The dimension statement follows from surjectivity.
\end{proof}

\begin{lemma}\label{lem:integralimage}
Choose a symplectic basis $a_1,b_1,\ldots,a_g,b_g$ so that the covering
character sends $a_1$ to $1\in\mathbb Z/p$ and the other basis elements to
zero.  Then
\[
\operatorname{im}\bigl(q_*:H_1(F_{g,p};\mathbb Z)\to
H_1(\Sigma_g;\mathbb Z)\bigr)
=\langle pa_1,b_1,a_2,b_2,\ldots,a_g,b_g\rangle.
\]
In particular, this image is the kernel of the covering character on
$H_1(\Sigma_g;\mathbb Z)$ and its cokernel is $\mathbb Z/p$.
\end{lemma}

\begin{proof}
Every closed loop upstairs projects to a loop in the kernel of the covering
character.  Conversely, the curves representing $b_1,a_i,b_i$ for $i\geq2$
lift to closed curves, as does the $p$-fold traversal of $a_1$.  Their homology
classes generate the displayed kernel, which proves the assertion.
\end{proof}

\begin{lemma}\label{lem:sympposition}
Let $V$ be a symplectic vector space of dimension $2G$, and let
$K_1,K_2\subset V$ be symplectic subspaces of dimensions $r$ and $s$.
There is a symplectic automorphism $T$ such that
\[
\dim(K_1\cap T^{-1}K_2)=\max\{0,r+s-2G\}.
\]
\end{lemma}

\begin{proof}
Use the symplectic orthogonal decomposition
\[
V=K_1\oplus K_1^\perp.
\]
If $s\leq 2G-r$, choose a symplectic subspace $K_2'\subset K_1^\perp$
of dimension $s$; then $K_1\cap K_2'=0$.  If $s>2G-r$, choose a
symplectic subspace $L\subset K_1$ of dimension $s-(2G-r)$ and put
\[
K_2'=K_1^\perp\oplus L.
\]
Then $K_2'$ is symplectic of dimension $s$ and
\[
\dim(K_1\cap K_2')=s-(2G-r)=r+s-2G.
\]
Every two symplectic subspaces of the same dimension are carried to one
another by a symplectic automorphism.  Choose one carrying $K_2'$ to
$K_2$, and take its inverse for $T$.
\end{proof}

Over $\Q$, put
\[
V=H_1(\Sigma_G;\Q),\qquad
A=Q_{p,g},\qquad B=Q_{q,h}.
\]
For a surface gluing inducing $T\in \Sp(V)$, the surface part of the Mayer--Vietoris map is
\[
\Theta_T=(A,-BT):V\longrightarrow H_1(\Sigma_g;\Q)\oplus H_1(\Sigma_h;\Q).
\]

\begin{lemma}\label{lem:gap}
For every numerical type in Theorem~\ref{thm:classification},
\[
G-g-h=t-1=\frac{k-1}{2}.
\]
In particular, $G\geq g+h$, with equality exactly when $k=1$.
\end{lemma}

\begin{proof}
Using $G=n+1$, $g=a+1$, $h=b+1$, and $n=a+b+t$ gives
\[
G-g-h=(a+b+t+1)-(a+b+2)=t-1.
\]
\end{proof}

\begin{theorem}\label{thm:existence}
Every numerical type in Theorem~\ref{thm:classification} admits an equatorially product-framed symplectic gluing $\varphi$ for which
\[
H_1(Z^{p,q}_{g,h}(\varphi);\Q)=0.
\]
Consequently
\[
H^*(Z^{p,q}_{g,h}(\varphi);\Q)
\cong
H^*(\#_k(S^2\times S^2);\Q)
\]
as graded rings.
\end{theorem}

\begin{proof}
By Lemma~\ref{lem:covermap},
\[
K_A=\ker A,\qquad K_B=\ker B
\]
are symplectic subspaces of $V$, of dimensions
\[
2G-2g,\qquad 2G-2h.
\]
Lemma~\ref{lem:sympposition} gives $T\in\Sp(V)$ for which
\[
\dim(K_A\cap T^{-1}K_B)
=
\max\{0,2G-2g-2h\}.
\]
By Lemma~\ref{lem:gap}, $G\geq g+h$, and therefore
\[
\dim(K_A\cap T^{-1}K_B)=2G-2g-2h.
\]
It follows that
\[
\rk\Theta_T=2g+2h,
\]
so $\Theta_T$ is onto.
Indeed,
\[
\ker\Theta_T=K_A\cap T^{-1}K_B,
\]
and therefore
\[
\rk\Theta_T=2G-(2G-2g-2h)=2g+2h,
\]
the dimension of the target.

The maximal-rank condition is the nonvanishing of at least one $(2g+2h)\times(2g+2h)$ minor.  Hence it is a nonempty Zariski-open condition on the symplectic group.  Since $\Sp(2G,\Z)$ is Zariski dense in $\Sp(2G,\R)$ \cite{Borel}, an integral symplectic matrix can be chosen.  The mapping class group of $\Sigma_G$ surjects onto $\Sp(2G,\Z)$ \cite{FarbMargalit}, so this matrix is realized by a gluing diffeomorphism.  The two induced area forms on $\Sigma_G$ have the same total area.  By Moser's theorem the diffeomorphism may be isotoped, without changing its action on homology, so that it preserves the area form.  Together with the reversal of the normal circle, it therefore gives a symplectic sum gluing.

Over $\Q$ the meridional torsion in \eqref{eq:H1comp} vanishes.  Surjectivity of $\Theta_T$ therefore gives $H_1(Z;\Q)=0$.  Since
\[
e(Z)=2+2k,\qquad \sigma(Z)=0,
\]
we obtain
\[
b_2(Z)=2k,\qquad b_2^+(Z)=b_2^-(Z)=k.
\]
The integral intersection pairing on $H_2(Z;\Z)/\Tor$ is unimodular of rank $2k$ and signature zero.  If it is odd, the classification of indefinite odd unimodular forms gives $k\langle1\rangle\oplus k\langle-1\rangle$; if it is even, the indefinite even unimodular form of signature $(k,k)$ is $kH$; see \cite[Chapter~II, Sections~4--5]{MilnorHusemoller}.  After tensoring with $\Q$, both forms are isometric to $kH$: for one pair $\langle1\rangle\oplus\langle-1\rangle$ with basis $e,f$, the vectors $e+f$ and $(e-f)/2$ form a hyperbolic pair.  Poincar\'e duality gives $H_3(Z;\Q)\cong H_1(Z;\Q)=0$, and for a closed oriented four-manifold with $b_1=0$ the graded rational cohomology ring is determined by its intersection pairing.  This proves the stated ring isomorphism.
\end{proof}

\begin{remark}
For $k=1$ the map $\Theta_T$ is square.  For $k>1$ it is rectangular, with
\[
\dim V-(2g+2h)=k-1.
\]
Thus the determinant used in the smallest case is replaced by maximal minors.
\end{remark}

\section{Integral homology}

The rational calculation forgets the finite meridional classes in the two complements.  At the integral level these classes are essential.  We therefore keep the splitting
\[
H_1(M_{g,p};\mathbb Z)\cong H_1(\Sigma_g;\mathbb Z)\oplus\langle\mu_p\rangle,
\qquad p\mu_p=0,
\]
and the analogous splitting on the second side.  The surface part of the boundary map is the covering push-forward $Q_{p,g}$, while the normal circle maps to the meridian.  A change of normal framing can add a meridional component to a surface class.

Choose the adapted equatorial normal framings of \cite{AkhmedovBraided}.  In these coordinates the boundary gluing has the form
\[
\varphi_*(v,r)=(Sv,\lambda(v)-r),
\]
where $S\in\Sp(2G,\Z)$ is the action on the surface homology and $\lambda:H_1(\Sigma_G;\Z)\to\Z$ records the framing change.  Thus the four column blocks in the presentation below correspond to the boundary surface classes, the common meridian, and the two torsion relations $p\mu_p=0$ and $q\mu_q=0$.

\begin{lemma}\label{lem:boundarymap}
With the equatorial framing, the inclusion
\(F_{g,p}\times S^1\hookrightarrow M_{g,p}\) induces
\[
 (v,r)\longmapsto (Q_{p,g}v,r\mu_p)
\]
under the splitting \eqref{eq:H1comp}.  Replacing the framing by a map
\(\lambda:H_1(F_{g,p};\mathbb Z)\to\mathbb Z\) changes only the
meridional component, replacing \(r\) by \(r+\lambda(v)\).
\end{lemma}

\begin{proof}
The projection of a boundary surface loop to the base is the covering
push-forward \(Q_{p,g}\).  In the equatorial trivialization its normal
coordinate is constant, so it has no meridional component.  More
geometrically, choose along the equator the unit normal vector field in
$S^2$ orthogonal to the equator.  Pushing a loop of $F_{g,p}$ in this
direction keeps it on one fixed side of the multisection, and the resulting
push-off has zero meridional coordinate, equivalently zero linking in the
cyclic meridian group $\mathbb Z/p$.  Thus its class has no component along
the meridian.  The normal circle bounds a normal disk in $\nu F_{g,p}$
and represents the peripheral loop \(\mu_p\) in the punctured-sphere
fiber of the complement.  Twisting the
normal trivialization along a loop \(v\) adds \(\lambda(v)\) turns around
that peripheral loop, which gives the last assertion.
\end{proof}

\begin{proposition}\label{prop:MV}
With rows ordered as
\[
H_1(\Sigma_g;\mathbb Z)\oplus\mathbb Z\langle\mu_p\rangle
\oplus H_1(\Sigma_h;\mathbb Z)\oplus\mathbb Z\langle\mu_q\rangle
\]
and columns ordered as
\[
H_1(\Sigma_G;\mathbb Z)\oplus\mathbb Z\langle r\rangle
\oplus\mathbb Z\langle\rho_p\rangle\oplus\mathbb Z\langle\rho_q\rangle,
\]
where $r$ is the common normal-circle generator and $\rho_p,\rho_q$
record the two meridional torsion relations, the first homology of the
mixed sum is the cokernel of
\[
\begin{pmatrix}
Q_{p,g}&0&0&0\\
0&1&p&0\\
-Q_{q,h}S&0&0&0\\
-\lambda&1&0&q
\end{pmatrix}.
\]
Equivalently, if $d=\gcd(p,q)$, it is the cokernel of the homomorphism
\[
H_1(\Sigma_G;\mathbb Z)\oplus\mathbb Z
\longrightarrow
H_1(\Sigma_g;\mathbb Z)\oplus H_1(\Sigma_h;\mathbb Z)\oplus\mathbb Z
\]
represented by
\[
\begin{pmatrix}
Q_{p,g}&0\\
-Q_{q,h}S&0\\
\lambda&d
\end{pmatrix}.
\]
For a product-framed gluing, $\lambda=0$, and
\[
H_1(Z;\Z)\cong
\coker\begin{pmatrix}Q_{p,g}\\-Q_{q,h}S\end{pmatrix}
\oplus \Z/d.
\]
\end{proposition}

\begin{proof}
The boundary and the two complements are connected, so the map on
$H_0$ in the Mayer--Vietoris sequence is injective.  Hence $H_1(Z)$ is
the cokernel of the boundary map on $H_1$.
By Lemma~\ref{lem:boundarymap}, the Mayer--Vietoris map on the boundary
surface and normal-circle generators is
\[
(v,r)\longmapsto
\bigl(Q_{p,g}v,r\mu_p,-Q_{q,h}Sv,
(r-\lambda(v))\mu_q\bigr),
\]
up to simultaneously changing the signs in one summand.  Adjoining the
two torsion relations gives the displayed four-block presentation from the
two splittings
\[
H_1(M_{g,p};\Z)=H_1(\Sigma_g;\Z)\oplus \Z/p,
\]
\[
H_1(M_{h,q};\Z)=H_1(\Sigma_h;\Z)\oplus \Z/q.
\]
The column containing the common normal circle gives
$\mu_p+\mu_q=0$.  Substituting $\mu_q=-\mu_p$ changes the entry
$-\lambda$ in the last row to $+\lambda$ in the reduced meridional row;
equivalently, this is an allowed integral sign operation followed by
elimination of the common normal-circle column.  The two remaining relations
\(p\mu=0\) and \(q\mu=0\) are equivalent, by integral elementary
operations, to \(d\mu=0\), where \(d=\gcd(p,q)\).  This yields the reduced
presentation.  When \(\lambda=0\), the meridional row is independent of
the surface rows and splits off as \(\mathbb Z/d\).
\end{proof}

\begin{corollary}\label{cor:quotients}
Every direct mixed sum has quotients $\Z/p$ and $\Z/q$ in first homology.  In particular, no direct mixed sum in this paper is simply connected or has the integral homology of $\#_k(S^2\times S^2)$.
\end{corollary}

\begin{proof}
By Lemma~\ref{lem:integralimage}, the image of $H_1(F_{g,p};\Z)$ in the base homology is the kernel of the covering character
\[
H_1(\Sigma_g;\Z)\to \Z/p.
\]
Project the target of the Mayer--Vietoris map to
\[
H_1(\Sigma_g;\mathbb Z)/\operatorname{im}Q_{p,g}\cong\mathbb Z/p
\]
and send the second summand and both meridional classes to zero.  Every
boundary relation, including the framing term, maps to zero, so the map
descends to a surjection from \(H_1(Z;\mathbb Z)\).  Reversing the roles of
the two sides gives the quotient \(\mathbb Z/q\).
\end{proof}

For $k>1$, a complete description of all finite abelian groups obtained
from the rectangular matrix is a separate integral problem.  For a
product-framed gluing, let
\[
A_S=\begin{pmatrix}Q_{p,g}\\-Q_{q,h}S\end{pmatrix}.
\]
If $A_S$ has full row rank, then
\[
|\operatorname{coker}A_S|
=\gcd\{\text{maximal minors of }A_S\},\qquad
|H_1(Z;\mathbb Z)|=\gcd(p,q)\,|\operatorname{coker}A_S|.
\]
For a nonzero framing term one must instead use the maximal minors of
the full reduced presentation in Proposition~\ref{prop:MV}.  This differs
from the square determinant calculation in the $k=1$ case.

\section{Boundary incompressibility, Bass--Serre splittings, and lattice obstructions}

The fundamental-group obstruction is independent of the numerical size.

\begin{lemma}\label{lem:productcover}
Let $q:F_{g,p}\to\Sigma_g$ be the explicit cyclic multisection cover of
Proposition~\ref{prop:multi}.  The pullback of the complement bundle admits
a natural fiberwise trivialization after the defining root coordinate and
equivariant tubular neighborhoods are fixed:
\[
q^*M_{g,p}\cong F_{g,p}\times P_p.
\]
\end{lemma}

\begin{proof}
A point of $F_{g,p}$ is a pair $(x,z)$ with $z^p=u(x)$.  Over $x$, the
punctures of the sphere fiber are the points
\[
z,\zeta z,\ldots,\zeta^{p-1}z,
\qquad \zeta=e^{2\pi i/p},
\]
on the equator.  Rotation of the sphere by multiplication with $z^{-1}$
fixes the two poles and carries this configuration to the fixed set of
$p$th roots of unity.  These rotations vary smoothly with $(x,z)$ and,
after choosing equivariant tubular disks around the punctures, give a
fiber-preserving diffeomorphism from the pullback complement to
$F_{g,p}\times P_p$.  Thus no residual pure braid or point-pushing
monodromy remains in this explicit root construction.
\end{proof}

\begin{theorem}\label{thm:incomp}
For the explicit cyclic equatorial multisection
$F_{g,p}\subset\Sigma_g\times S^2$ of degree $p\geq2$, the homomorphism
\[
\pi_1(\partial\nu F_{g,p})\longrightarrow \pi_1(M_{g,p})
\]
is injective.
\end{theorem}

\begin{proof}
Let $\chi:\pi_1(\Sigma_g)\to\mathbb Z/p$ be the covering character and
pull $M_{g,p}$ back to the corresponding cover.  By
Lemma~\ref{lem:productcover}, the resulting complement cover is
$F_{g,p}\times P_p$.  The character defining this cover is trivial on
the boundary group: $q_*$ is injective,
$q_*\pi_1(F_{g,p})=\ker\chi$, and the normal circle is vertical.  More
explicitly, the pullback of the boundary is
\[
F_{g,p}\times_{\Sigma_g}F_{g,p}\times S^1.
\]
Since $q$ is a regular cyclic cover, the fiber product is the disjoint
union of the graphs of its $p$ deck transformations.  Thus each lifted
boundary component is naturally $F_{g,p}\times S^1$ and maps
homeomorphically to the original boundary.
For such a component the inclusion is
\[
\pi_1(F_{g,p})\times\mathbb Z\longrightarrow
\pi_1(F_{g,p})\times\pi_1(P_p),
\]
and is injective because a boundary circle of $P_p$ is incompressible for
$p\geq2$.  If a boundary loop is null-homotopic in $M_{g,p}$, its closed
lift to this boundary component has a lifted null-homotopy in the product
cover.  Injectivity upstairs makes the lifted loop trivial.  Since the
chosen boundary component maps homeomorphically to the original one, the
original loop is trivial as well.
\end{proof}

\begin{corollary}\label{cor:pi1}
For every direct mixed sum in Theorem~\ref{thm:classification},
\[
\pi_1(\Sigma_G)\times\Z\hookrightarrow \pi_1(Z).
\]
Consequently $\pi_1(Z)$ is nontrivial and contains $\Z^2$.
\end{corollary}

\begin{proof}
Van Kampen identifies the group with an amalgamated product of the two complement groups over the common boundary group.  By Theorem~\ref{thm:incomp}, both boundary maps are injective, and the normal form theorem for amalgamated products injects the boundary group into the final group.
\end{proof}

\begin{proposition}\label{prop:bassserre}
For every direct mixed sum, $\pi_1(Z)$ acts without inversion on a Bass--Serre tree $T$ whose quotient is a single edge.  The two vertex stabilizers are conjugate to $\pi_1(M_{g,p})$ and $\pi_1(M_{h,q})$, and the edge stabilizer is conjugate to $\pi_1(\Sigma_G)\times\mathbb Z$.  Moreover, writing $F_r$ for the free group of rank $r$,
\[
\pi_1(M_{g,p})\text{ contains an index-}p\text{ subgroup }\pi_1(\Sigma_G)\times F_{p-1},
\]
and similarly $\pi_1(M_{h,q})$ contains an index-$q$ subgroup $\pi_1(\Sigma_G)\times F_{q-1}$.  The splitting is nontrivial.
\end{proposition}

\begin{proof}
Van Kampen and Theorem~\ref{thm:incomp} give
\[
\pi_1(Z)\cong
\pi_1(M_{g,p})*_{\pi_1(\Sigma_G)\times\mathbb Z}\pi_1(M_{h,q}).
\]
This gives the Bass--Serre action and the stabilizers.  Lemma~\ref{lem:productcover} identifies the pullback of $M_{g,p}\to\Sigma_g$ by the cyclic cover $\Sigma_G\to\Sigma_g$ with $\Sigma_G\times P_p$.  Hence the corresponding index-$p$ subgroup is
\[
\pi_1(\Sigma_G)\times\pi_1(P_p)=\pi_1(\Sigma_G)\times F_{p-1}.
\]
The same argument applies on the other side.  Finally, each boundary
inclusion is proper.  On the first side, projection to the base quotient
gives
\[
H_1(M_{g,p};\mathbb Z)/\operatorname{im}H_1(\partial M_{g,p};\mathbb Z)
\twoheadrightarrow\mathbb Z/p,
\]
and the analogous quotient on the second side surjects onto \(\mathbb Z/q\).
Thus neither vertex group equals the edge group, and the splitting is
nontrivial.
\end{proof}

The resulting lattice obstructions are independent of $k$.  Here a
lattice means a discrete subgroup of finite covolume.  For a lattice in
$\operatorname{PSL}_2(\mathbb R)\times\operatorname{PSL}_2(\mathbb R)$,
irreducibility means that its projection to each factor is dense.

\begin{lemma}\label{lem:centralizer}
Let $\Gamma<\operatorname{PSL}_2(\mathbb R)\times\operatorname{PSL}_2(\mathbb R)$.  Suppose that $\Gamma$ contains a subgroup $H\times\langle z\rangle$, where $H$ is a nonabelian surface group and $z$ has infinite order.  If both projections of $z$ to the two factors are nontrivial, then this is impossible.
\end{lemma}

\begin{proof}
Write $z=(z_1,z_2)$.  Since $z$ centralizes $H$, the image of $H$ lies in
\[
C_{\operatorname{PSL}_2(\mathbb R)}(z_1)\times C_{\operatorname{PSL}_2(\mathbb R)}(z_2).
\]
The centralizer of a nonidentity element of $\operatorname{PSL}_2(\mathbb R)$ is abelian.  Indeed, according as the element is hyperbolic, parabolic, or elliptic, its centralizer is conjugate to a one-parameter diagonal subgroup, a one-parameter unipotent subgroup, or the rotation subgroup.  If both $z_1$ and $z_2$ are nontrivial, the displayed product is abelian, contradicting the fact that $H$ is nonabelian.
\end{proof}

\begin{theorem}\label{thm:lattice}
Let $Z$ be a direct mixed sum formed from the explicit cyclic equatorial
multisections of Section~2.  Then:
\begin{enumerate}[label=(\arabic*)]
\item $\pi_1(Z)$ is not word-hyperbolic;
\item $\pi_1(Z)$ is not a cocompact lattice in $\operatorname{PO}(n,1)$ for any $n\geq2$;
\item $\pi_1(Z)$ is not a cocompact lattice in $\operatorname{PU}(n,1)$ for any $n\geq2$;
\item $\pi_1(Z)$ is not a torsion-free cocompact irreducible lattice in $\operatorname{PSL}_2(\mathbb R)\times\operatorname{PSL}_2(\mathbb R)$.
\end{enumerate}
In particular, it is not a cocompact Deligne--Mostow lattice in $\operatorname{PU}(2,1)$.
\end{theorem}

\begin{proof}
By Corollary~\ref{cor:pi1}, the group $\pi_1(Z)$ contains
$\pi_1(\Sigma_G)\times\mathbb Z$, and hence contains $\mathbb Z^2$.
Every abelian subgroup of a word-hyperbolic group is virtually cyclic.
It follows that $\pi_1(Z)$ is not word-hyperbolic, proving (1).

Let $X$ be either real hyperbolic $n$-space or complex hyperbolic
$n$-space.  A cocompact lattice in $\operatorname{Isom}(X)$ acts properly
and cocompactly on $X$.  By the Milnor--\v{S}varc lemma it is
quasi-isometric to $X$, and is therefore word-hyperbolic because $X$ is a
proper geodesic Gromov-hyperbolic space; see \cite{BridsonHaefliger}.
Thus (1) implies (2) and (3).

For (4), suppose that $\pi_1(Z)$ is isomorphic to a torsion-free
cocompact irreducible lattice
\[
\Gamma<\operatorname{PSL}_2(\mathbb R)\times
\operatorname{PSL}_2(\mathbb R).
\]
Let $\operatorname{pr}_i$ denote projection to the $i$th factor.  By
irreducibility, both $\operatorname{pr}_1(\Gamma)$ and
$\operatorname{pr}_2(\Gamma)$ are dense.  Set
\[
K_1=\Gamma\cap
\bigl(\operatorname{PSL}_2(\mathbb R)\times\{1\}\bigr).
\]
Then $K_1$ is discrete and normal in $\Gamma$.  Its projection to the
first factor is normalized by the dense subgroup
$\operatorname{pr}_1(\Gamma)$.  Since the normalizer of a closed subgroup
is closed, $\operatorname{pr}_1(K_1)$ is normal in all of
$\operatorname{PSL}_2(\mathbb R)$.  It is therefore a discrete normal
subgroup of the connected center-free simple group
$\operatorname{PSL}_2(\mathbb R)$, and hence is trivial.  Thus $K_1$ is
trivial.  The same argument for the other factor shows that both
coordinate projections are injective on $\Gamma$.

Under the assumed isomorphism, Corollary~\ref{cor:pi1} gives a subgroup
$\pi_1(\Sigma_G)\times\langle z\rangle$ of $\Gamma$, where $z$ has
infinite order.  Here $G=p(g-1)+1\geq3$, so $\pi_1(\Sigma_G)$ is
nonabelian.  The injectivity of the two coordinate projections implies
that both coordinates of $z$ are nontrivial, contradicting
Lemma~\ref{lem:centralizer}.  This proves (4).  Finally, any
Deligne--Mostow lattice in $\operatorname{PU}(2,1)$ is a
complex-hyperbolic lattice, so its cocompact instances are excluded by
(3) \cite{DeligneMostow}.
\end{proof}

\subsection{Real and complex hyperbolic hybrids}
Gromov and Piatetski-Shapiro \cite{GPS} constructed nonarithmetic lattices in real-hyperbolic isometry groups by cutting arithmetic hyperbolic manifolds along totally geodesic hypersurfaces and gluing pieces coming from noncommensurable arithmetic manifolds.  On fundamental groups their construction is governed by an amalgam over the hypersurface group.  The direct braided sums above also give an amalgam,
\[
\pi_1(Z)=\pi_1(M_{g,p})*_{\pi_1(\Sigma_G)\times\mathbb Z}\pi_1(M_{h,q}),
\]
but the edge group contains $\mathbb Z^2$.  Theorem~\ref{thm:lattice} therefore rules out a cocompact real-hyperbolic lattice.  The common feature is only the graph-of-groups decomposition, not a hyperbolic uniformization.

There is a parallel comparison in complex hyperbolic geometry.  In
complex dimension two, Deligne and Mostow \cite{DeligneMostow}
constructed lattices in $\operatorname{PU}(2,1)$ from hypergeometric
monodromy, including nonarithmetic examples.  Falbel and Pasquinelli
proved that the five nonarithmetic complex two-dimensional
Deligne--Mostow lattices treated in their theorem can be
realized as hybrids generated by two noncommensurable arithmetic
Fuchsian lattices in $\operatorname{PU}(1,1)$
\cite[Theorem~4.2]{FalbelPasquinelli}.  For the three smallest mixed sums,
the degree data give the Euclidean triangle triples $(2,3,6)$,
$(2,4,4)$, and $(3,3,3)$; see \cite{AkhYeung}.  The higher cases with the
same numerical relation are described as follows.

\begin{proposition}\label{prop:triangletypes}
Let $k=2t-1$ and suppose that $t\mid n$.  Put $r=n/t$.  Up to
interchanging $p$ and $q$, the numerical types satisfying
\[
\frac1p+\frac1q+\frac1r=1
\]
are exactly the following:
\[
\begin{array}{c|c|c|c}
(p,q,r)&n&(g-1,h-1)&\text{condition}\ \\ \hline
(3,3,3)&3t&(t,t)&\text{all }t\\
(4,2,4)&4t&(t,2t)&\text{all }t\\
(4,4,2)&2t&(t/2,t/2)&2\mid t\\
(3,2,6)&6t&(2t,3t)&\text{all }t\\
(6,2,3)&3t&(t/2,3t/2)&2\mid t\\
(6,3,2)&2t&(t/3,2t/3)&3\mid t.
\end{array}
\]
\end{proposition}

\begin{proof}
The classification equation may be written as
\[
\frac1p+\frac1q+\frac{t}{n}=1.
\]
Whenever $t\mid n$, write $n=tr$.  Then
\[
\frac1p+\frac1q+\frac1r=1.
\]
Since $p,q,r\geq2$, the triple $(p,q,r)$ is a permutation of
\[
(2,3,6),\qquad (2,4,4),\qquad (3,3,3).
\]
After ordering the three integers, the smallest is at most three; the
cases two and three give exactly the triples above.  Since $r$ is
distinguished, placing it in each of the possible positions gives the six
rows of the table.  A placement gives a numerical type precisely when
\[
g-1=\frac{tr}{p},\qquad h-1=\frac{tr}{q}
\]
are integers.  These are exactly the divisibility conditions displayed
in the last column.
\end{proof}

Proposition~\ref{prop:triangletypes} gives only a numerical connection.  It
supplies neither a homomorphism from the mixed-sum group to a Euclidean
triangle group nor an identification of the cyclic covering monodromy
with a triangle-group representation or with Deligne--Mostow
hypergeometric monodromy.  Moreover, these Euclidean triangle groups are
different from the hyperbolic Fuchsian triangle lattices used in the
Falbel--Pasquinelli hybrid construction.  We will study the possible
geometric significance of this relation in a subsequent work.  Since a
cocompact complex-hyperbolic lattice is word-hyperbolic and every direct
mixed-sum group contains $\mathbb Z^2$, the present construction cannot
produce such a lattice.

The cocompact qualification is essential in the Deligne--Mostow comparison.
Nonuniform complex-hyperbolic lattices have parabolic subgroups, so the
$\mathbb Z^2$ argument alone does not exclude them.  The bidisk has higher
rank and products of surface groups contain $\mathbb Z^2$; consequently,
the obstruction to an irreducible bidisk lattice requires the centralizer
argument of Lemma~\ref{lem:centralizer}.  This argument does not exclude
reducible bidisk lattices.

\section{Minimality}

\begin{theorem}\label{thm:min}
Every symplectic direct mixed sum considered in this paper is symplectically minimal.
\end{theorem}

\begin{proof}
Each summand $\Sigma_g\times S^2$, with $g\geq 2$, is a minimal ruled surface.  The gluing surface $F_{g,p}$ has genus
\[
p(g-1)+1>g
\]
for $p\geq 2$.  It cannot be a section of a ruling, since a section has the genus of the base.  Usher's minimality theorem for symplectic sums \cite{Usher} therefore applies.  Indeed, neither summand contains an exceptional sphere, and the ruled exception in Usher's theorem occurs only when the sum surface is a section.  Both exceptional alternatives are absent.
\end{proof}

\begin{corollary}
Every rational-cohomology example of Theorem~\ref{thm:existence} has symplectic Kodaira dimension two.
\end{corollary}

\begin{proof}
The examples are minimal.  They are not rational, since rational
four-manifolds are simply connected whereas Corollary~\ref{cor:pi1}
gives a nontrivial fundamental group containing $\mathbb Z^2$.  They are
not ruled over a positive-genus surface, since such a ruled manifold has
positive first Betti number whereas $b_1(Z)=0$.  They have
\[
K_\omega^2=2e+3\sigma=4(k+1)>0.
\]
Since the manifold is minimal, nonrational, and nonruled, the main theorem of Liu \cite{Liu}, together with $K_\omega^2>0$, gives $K_\omega\cdot[\omega]>0$.
\end{proof}

\section{Spin gluings}

We prove that every higher numerical type has a spin realization.  A
spin structure on a closed oriented surface will be identified with its
quadratic refinement on $H_1(-;\Z/2)$.  The normal framings in this
section are chosen to be compatible with the prescribed spin structures;
they need not agree with a fixed equatorial framing.

\begin{lemma}\label{lem:spinpull}
Let $q:\Sigma_G\to\Sigma_g$ be a connected cyclic covering of degree $p$
defined by a character of exact order $p$.
As the spin structure $s$ on $\Sigma_g$ varies, both Arf invariants occur
among the pullbacks $q^*s$.
\end{lemma}

\begin{proof}
Choose a complex structure downstairs and lift it to the cover.  Represent
$s$ by a theta characteristic $L$.  If $A$ is the flat line bundle of
exact order $p$ defining the cyclic cover, then, up to replacing $A$ by
its inverse,
\[
q_*\mathcal O_{\Sigma_G}=\bigoplus_{j=0}^{p-1}A^j.
\]
The projection formula gives
\[
h^0(\Sigma_G,q^*L)=\sum_{j=0}^{p-1}h^0(\Sigma_g,L\otimes A^j).
\]
Serre duality, Riemann--Roch, and $L^2=K$ identify modulo two the terms
indexed by $j$ and $-j$.
They therefore cancel in pairs modulo two.  If $p$ is odd, only the
trivial character remains, and Atiyah's mod-two index theorem
\cite{AtiyahSpin} gives
\[
\operatorname{Arf}(q^*s)=\operatorname{Arf}(s).
\]
Both values occur among spin structures on $\Sigma_g$.

Suppose that $p$ is even.  The two unpaired characters are the trivial
character and the nontrivial order-two character $\eta=A^{p/2}$.  We also write
$\eta$ for the corresponding nonzero class in
$H^1(\Sigma_g;\mathbb Z/2)$.  If $q_s$ is the
quadratic refinement associated to $s$, then
\[
q_s(\operatorname{PD}\eta)
\equiv h^0(\Sigma_g,L\otimes\eta)+h^0(\Sigma_g,L)\pmod2;
\]
this is the theta-characteristic definition of the quadratic refinement
\cite{MumfordTheta}.  The projection formula and the cancellation of the
paired characters therefore give
\[
\operatorname{Arf}(q^*s)=q_s(\operatorname{PD}\eta).
\]
For $\alpha\in H^1(\Sigma_g;\mathbb Z/2)$,
\[
q_{s+\alpha}(\operatorname{PD}\eta)
=q_s(\operatorname{PD}\eta)
+\langle\alpha\smile\eta,[\Sigma_g]\rangle.
\]
Since $\eta\ne0$ and the intersection pairing is nondegenerate, the last
term takes both values.  Hence both pullback parities occur.  The mod-two
index is topological, so the result is independent of the auxiliary
complex structure.
\end{proof}

\begin{lemma}\label{lem:spinframing}
Let $X$ be an oriented spin four-manifold and let $F\subset X$ be an
oriented embedded surface with trivial oriented normal bundle.  Given a
spin structure $s_F$ on $F$, there is an oriented normal framing for which
the spin structure induced on $TF$ from
\[
TX|_F=TF\oplus\nu F
\]
is $s_F$.  In the corresponding tubular coordinates the boundary spin
structure on $F\times S^1$ is
$s_F\boxtimes s_{\mathrm{bnd}}$, where $s_{\mathrm{bnd}}$ is the
bounding spin structure on the circle.  If $F$ is symplectic and
$F^2=0$, the framing may be chosen symplectic.
\end{lemma}

\begin{proof}
Start with a normal framing inducing a spin structure $s_0$ on $F$.  The
difference $s_F-s_0$ is a class $\delta\in H^1(F;\mathbb Z/2)$.  The
reduction map $H^1(F;\mathbb Z)\to H^1(F;\mathbb Z/2)$ is onto.  Choose
an integral lift of $\delta$ and represent it by a map $F\to SO(2)$.
Twisting the framing by this map changes the induced tangential spin
structure by $\delta$.  The spin structure on the normal disk restricts
to the bounding spin structure on its boundary.  In the symplectic case
the same twist lies in $SO(2)=U(1)$ and preserves the symplectic normal
trivialization.
\end{proof}

\begin{theorem}\label{thm:spin}
For every numerical type in Theorem~\ref{thm:classification}, there are
spin-compatible symplectic normal framings and a product gluing relative
to those framings such that
\[
H_1(Z;\Q)=0,\qquad w_2(Z)=0.
\]
For such a gluing,
\[
Q_{Z,\mathrm{free}}\cong kH.
\]
\end{theorem}

\begin{proof}
Choose spin structures $s_g$ and $s_h$ on the base surfaces so that their
pullbacks to $\Sigma_G$ have the same Arf invariant.  This is possible by
Lemma~\ref{lem:spinpull}.  Equip the two products with the corresponding
product spin structures.  We retain these pullback spin structures because
they record the restrictions of the product spin structures on the two ruled
summands; Lemma~\ref{lem:spinframing} then changes only the normal framings.
By that lemma, choose
symplectic normal framings which induce $q_g^*s_g$ and $q_h^*s_h$ on the
two multisections.  The boundary spin structures are then
\[
q_g^*s_g\boxtimes s_{\mathrm{bnd}},\qquad
q_h^*s_h\boxtimes s_{\mathrm{bnd}}.
\]

The group $\Sp(2G,\Z/2)$ is transitive on quadratic refinements of fixed
Arf invariant \cite{MumfordTheta}.  Reduction
\[
\Sp(2G,\Z)\longrightarrow \Sp(2G,\Z/2)
\]
is onto: the elementary symplectic transvections generating the target
lift integrally.  The mapping class group surjects onto $\Sp(2G,\Z)$.
Write $\bar q_g$ and $\bar q_h$ for the two pulled-back quadratic
refinements.  Fix an integral symplectic matrix $T_0$ whose mod-two action
satisfies
\[
T_0^*\bar q_h=\bar q_g.
\]
Every matrix in $T_0\Gamma(2)$ has the same mod-two action.
Here
\[
\Gamma(2)=\ker\bigl(\Sp(2G,\mathbb Z)\longrightarrow
\Sp(2G,\mathbb Z/2)\bigr).
\]

Let $U\subset\Sp(2G,\R)$ be the locus on which
\[
(Q_{p,g},-Q_{q,h}T)
\]
has rank $2g+2h$.  The proof of Theorem~\ref{thm:existence} shows that
$U$ is nonempty and Zariski open.  The subgroup $\Gamma(2)$ has finite
index in $\Sp(2G,\mathbb Z)$ and is therefore Zariski dense.  Hence its
translate $T_0\Gamma(2)$ meets the nonempty Zariski-open set $U$.  Choose a
matrix in the intersection, realize it by a surface diffeomorphism, and
isotope that diffeomorphism to preserve an area form.

Glue this surface map to an orientation-reversing diffeomorphism of the
normal circle.  The surface map identifies the tangential spin
structures.  With the boundary orientations induced from the two
complements, the reversal of the normal circle is exactly the reversal
required by the gluing, and it preserves the bounding spin structure on
$S^1$.  The boundary spin structures therefore match and the complement
spin structures glue.  Hence $w_2(Z)=0$.  Since the chosen
matrix lies in $U$, Mayer--Vietoris gives $H_1(Z;\Q)=0$.

The free intersection form is even and unimodular, of rank $2k$ and
signature zero.  The indefinite even unimodular lattice of signature
$(k,k)$ is unique \cite[Chapter~II, Sections~4--5]{MilnorHusemoller},
hence it is $kH$.
\end{proof}

\begin{remark}
The spin conclusion concerns the closed four-manifold, not merely parity of the free lattice.  When $H_1$ has $2$-torsion, a nonzero $w_2$ can lie in the Ext term and be invisible on $H_2/\Tor$.  The theorem avoids this ambiguity by constructing a genuine spin structure.
\end{remark}

\subsection{The free intersection lattice and the canonical class}
For the spin realizations the integral lattice is completely determined.  This is the higher-$k$ analogue of the even lattice calculation in the rational cohomology quadric case.

\begin{proposition}\label{prop:canonical}
Let $Z$ be a spin realization from Theorem~\ref{thm:spin}.  Then
\[
Q_{Z,\mathrm{free}}\cong kH,
\qquad K_\omega^2=4(k+1).
\]
Moreover the image of the symplectic canonical class in $H^2(Z;\mathbb Z)/\Tor$ is characteristic and divisible by two.
\end{proposition}

\begin{proof}
The lattice statement is Theorem~\ref{thm:spin}.  Since $e(Z)=2+2k$ and $\sigma(Z)=0$,
\[
K_\omega^2=2e(Z)+3\sigma(Z)=4(k+1).
\]
The canonical class is characteristic, and on a spin four-manifold its
mod-two reduction is zero.  After passage to the torsion-free quotient,
its image therefore lies in $2(kH)$ and is divisible by two.
\end{proof}

Choose a hyperbolic basis $A_1,B_1,\ldots,A_k,B_k$ for $kH$.  Writing
\[
K_{\omega,\mathrm{free}}=2\sum_{i=1}^k(x_iA_i+y_iB_i)
\]
gives the numerical constraint
\[
2\sum_{i=1}^k x_i y_i=k+1.
\]
Unlike the rank-two case, this equation does not determine the canonical class up to lattice isometry without additional information.  For $k>1$, determining the lattice orbit of the canonical class is therefore an additional problem.

\section{Cyclic covers associated to the mixed-sum presentation}

For the smallest mixed sums of \cite{AkhYeung}, the common integer $n$ happened to equal $\operatorname{lcm}(p,q)$.  This need not hold in the higher classification, so we use the least common multiple as the covering degree.

Put
\[
\ell=\operatorname{lcm}(p,q).
\]
After fixing identifications of the two cyclic deck groups, let
\[
\alpha_p:\pi_1(M_{g,p})\to\Z/p,
\qquad
\alpha_q:\pi_1(M_{h,q})\to\Z/q
\]
be the characters obtained by composing the bundle projections to the
base groups with the two cyclic covering characters.

\begin{proposition}\label{prop:cover}
After fixing the two cyclic characters, identifications of their target
groups, and the standard injections into $\mathbb Z/\ell$, a direct
mixed-sum presentation determines a regular cyclic cover of degree $\ell$.
Over the first complement its inverse image is the disjoint union of $\ell/p$ copies of
\[
\Sigma_G\times P_p,
\]
and over the second complement it is the disjoint union of $\ell/q$ copies of
\[
\Sigma_G\times P_q.
\]
\end{proposition}

\begin{proof}
Define
\[
\beta_p=\frac{\ell}{p}\alpha_p:\pi_1(M_{g,p})\to\Z/\ell,
\]
and similarly define $\beta_q$.  The restriction of either character to the boundary subgroup is zero: the surface subgroup maps to the kernel of the covering character, and the meridian lies in the punctured-sphere fiber group.  Hence the two characters extend across the amalgamated product to a homomorphism
\[
\beta:\pi_1(Z)\to\Z/\ell.
\]
Here multiplication by $\ell/p$ means the standard injection
$\mathbb Z/p\hookrightarrow\mathbb Z/\ell$ sending $1$ to $\ell/p$.
The two images have orders $p$ and $q$.  In a cyclic group the subgroup
generated by subgroups of these orders has order
$\operatorname{lcm}(p,q)=\ell$, so $\beta$ is onto.  Its kernel is a
normal subgroup of index $\ell$ and therefore determines a connected
regular cyclic cover with deck group $\mathbb Z/\ell$.  The number of
components over the first complement is
$[\mathbb Z/\ell:\operatorname{im}\beta_p]=\ell/p$, and similarly it is
$\ell/q$ over the second.  Lemma~\ref{lem:productcover} identifies every
component with $\Sigma_G\times P_p$, respectively
$\Sigma_G\times P_q$.  The stated decomposition follows.
\end{proof}

\begin{remark}
For $k=1$, one has $\ell=n$ in all three numerical cases.  In higher cases the two numbers may differ.  For example, the $k=3$ type $(p,q,n)=(3,3,6)$ has $\ell=3$.
\end{remark}

\section{A conditional twisted ruled analogue}

Let $E_g\to\Sigma_g$ be the nontrivial oriented $S^2$-bundle.  Choose a section $S$ with $S^2=1$ and let $F$ be the sphere fiber.  A square-zero degree-$p$ class has the form
\[
A_{g,p}=pS-\frac p2F,
\]
so $p$ must be even.  Adjunction gives
\[
g(A_{g,p})=p(g-1)+1.
\]
Thus the same numerical classification applies, with the additional requirement that both degrees be even.
The calculation alone does not produce an unbranched symplectic
representative.  The result below is therefore conditional on the
existence of connected representatives with the stated covering
property.  When such a representative is chosen, we denote the embedded
surface by $\mathcal A_{g,p}\subset E_g$ and reserve $A_{g,p}$ for its
homology class.

\begin{proposition}\label{prop:twisted}
Suppose that a numerical type in Theorem~\ref{thm:classification} is
represented on both nontrivial ruled surfaces by connected unbranched
symplectic multisections $\mathcal A_{g,p}$ and $\mathcal A_{h,q}$ in the
square-zero classes $A_{g,p}$ and $A_{h,q}$.  Then $p$ and $q$ are both
even.  For such a representative,
\[
H_1(E_g\setminus\operatorname{int}\nu \mathcal A_{g,p};\Z)
\cong \Z^{2g}\oplus \Z/(p/2).
\]
Whenever both sides exist as connected unbranched symplectic multisections, a suitable gluing has vanishing rational first homology and hence the rational cohomology ring of $\#_k(S^2\times S^2)$.
\end{proposition}

\begin{proof}
The parity restriction follows from the square-zero equation.  Excision
and the Thom isomorphism give the exact segment
\[
H_2(E_g;\mathbb Z)\longrightarrow
H_2(E_g,E_g\setminus\nu \mathcal A_{g,p};\mathbb Z)\cong\mathbb Z
\longrightarrow H_1(E_g\setminus\nu \mathcal A_{g,p};\mathbb Z)
\longrightarrow H_1(E_g;\mathbb Z)\longrightarrow0.
\]
The first arrow is the intersection map.  Its image is generated by
\[
\gcd\left(p,\frac p2\right)=\frac p2,
\]
the divisibility of $A_{g,p}$.  Since
$H_1(E_g;\mathbb Z)\cong\mathbb Z^{2g}$ is free, the resulting extension
splits and gives the asserted complement homology.  The finite meridional
term disappears over $\Q$.  Because the multisections are unbranched,
their covering push-forwards remain surjective with symplectic kernels
over $\Q$.  The symplectic-position argument of
Theorem~\ref{thm:existence} therefore supplies a gluing for which the
rational Mayer--Vietoris map has full row rank.
\end{proof}

\begin{remark}
If $p$ is divisible by four, then $A_{g,p}\equiv0\pmod2$.  In the
mod-two pair sequence the kernel of
\[
H^2(E_g;\mathbb Z/2)\longrightarrow
H^2(E_g\setminus\nu \mathcal A_{g,p};\mathbb Z/2)
\]
is generated by $\operatorname{PD}(A_{g,p})$, hence is zero.  The nonzero
$w_2(E_g)$ therefore restricts nontrivially to the complement.  Any mixed
sum containing such a complement is nonspin.  This recovers the behavior
of the twisted degree-four side in the smallest $(2,4)$ construction.
\end{remark}

\section{The first two higher levels: \texorpdfstring{$3(S^2\times S^2)$ and $5(S^2\times S^2)$}{3(S2 x S2) and 5(S2 x S2)}}

The first two higher levels already differ from $k=1$.  The Mayer--Vietoris surface map is rectangular and
\[
G-g-h=\frac{k-1}{2}>0.
\]
These cases supply the integral, spin, and covering calculations below.

\begin{theorem}\label{thm:lowlevels}
At the $k=3$ level there are exactly five unordered numerical types, and at the $k=5$ level there are exactly seven.  Every one of these twelve types admits a minimal equatorially product-framed symplectic realization with the required rational cohomology.  It also admits a spin realization, with free intersection lattice $3H$ for $k=3$ and $5H$ for $k=5$.  None of the direct realizations is simply connected.  For every rational-cohomology realization the first homology is finite.  The persistent quotients $\mathbb Z/p$ and $\mathbb Z/q$ imply that its order is divisible by $\operatorname{lcm}(p,q)$.
\end{theorem}

\begin{proof}
The lists follow from Theorem~\ref{thm:param}.  Existence and minimality are Theorems~\ref{thm:existence} and~\ref{thm:min}; the spin statement is Theorem~\ref{thm:spin}.  Boundary incompressibility gives the nontrivial fundamental group.  By Corollary~\ref{cor:quotients}, a finite first homology group surjects onto both $\mathbb Z/p$ and $\mathbb Z/q$.  Its order is therefore divisible by both $p$ and $q$, and hence by $\operatorname{lcm}(p,q)$.
\end{proof}

\subsection{The five types for \texorpdfstring{$3(S^2\times S^2)$}{3(S2 x S2)}}
Here $t=2$.  There are exactly five unordered types:
\[
(4,4;2,2;5),\ (3,3;3,3;7),\ (6,2;2,4;7),\ (4,2;3,5;9),\ (3,2;5,7;13),
\]
where the notation is $(p,q;g,h;G)$ and the corresponding values of $n$ are $4,6,6,8,12$.

For every type, $e=8$ and $\sigma=0$, while $G=g+h+1$.  Hence a maximal-rank rational Mayer--Vietoris map has two-dimensional kernel.  Theorem~\ref{thm:existence} gives an equatorially product-framed gluing with
\[
b_1=0,\quad b_2=6,\quad b_2^+=b_2^-=3,
\]
and therefore the rational cohomology of $\#_3(S^2\times S^2)$.  The spin gluing of Theorem~\ref{thm:spin} has free lattice $3H$.

The symmetric $(4,4)$ type uses two copies of $\Sigma_2\times S^2$ and a genus-five degree-four multisection.  Its two complements have $H_1\cong\Z^4\oplus\Z/4$.  The symmetric $(3,3)$ type uses two copies of $\Sigma_3\times S^2$ and a genus-seven degree-three multisection.  The mixed $(6,2)$ type has the same gluing genus seven, but uses $\Sigma_2\times S^2$ and $\Sigma_4\times S^2$.  Thus the gluing genus alone does not determine the complement topology.

The $(4,2)$ type glues $\Sigma_3\times S^2$ to $\Sigma_5\times S^2$ along $\Sigma_9$, while the $(3,2)$ type glues $\Sigma_5\times S^2$ to $\Sigma_7\times S^2$ along $\Sigma_{13}$.  In every case the persistent quotients $\Z/p$ and $\Z/q$ show that the direct sum does not have the integral homology of $\#_3(S^2\times S^2)$.

The associated cyclic-cover degrees $\ell=\operatorname{lcm}(p,q)$ are respectively
\[
4,\ 3,\ 6,\ 4,\ 6.
\]
In particular, for the $(3,3)$ type one has $n=6$ but $\ell=3$.  This is the first level where the two natural covering integers separate.

\subsection{Three model constructions at the \texorpdfstring{$k=3$}{k=3} level}
Three of the five types will be examined: the symmetric even-degree case $(4,4)$, the symmetric odd-degree case $(3,3)$, and the asymmetric case $(6,2)$.  The last two have the same gluing genus but different complement topology.

For later use, recall the complement presentation.  If $x_1,\ldots,x_p$ are the peripheral loops of the $p$-punctured sphere and the covering character sends $a_1$ to $1\in\mathbb Z/p$ and the remaining standard generators to zero, then
\begin{align*}
\pi_1(M_{g,p})=\langle a_i,b_i,x_1,\ldots,x_p\mid {}&
\prod_{i=1}^g[a_i,b_i]=1,\quad x_1\cdots x_p=1,\\
&a_1x_ja_1^{-1}=x_{j+1},\quad [b_1,x_j]=1,\\
&[a_i,x_j]=[b_i,x_j]=1\quad (i\ge2)\rangle ,
\end{align*}
where the indices on the $x_j$ are read modulo $p$.  The north-pole section
fixes the fiber basepoint, so this is the semidirect-product presentation
$F_{p-1}\rtimes\pi_1(\Sigma_g)$ with literal cyclic rotation
$x_j\mapsto x_{j+1}$.  Its abelianization is
\[
H_1(M_{g,p};\mathbb Z)\cong\mathbb Z^{2g}\oplus\mathbb Z/p.
\]
The boundary group $\pi_1(\Sigma_G)\times\mathbb Z$ injects into this group.  Consequently the three constructions below are amalgamated products along the incompressible boundary subgroup.

The integral surface maps can be written uniformly.

\begin{lemma}\label{lem:integralbasis}
For the cyclic cover $\Sigma_G\to\Sigma_g$ determined by a primitive
character, there are symplectic bases
\[
(e_0,f_0),\qquad (e_{i,r},f_{i,r}),\quad 1\leq i\leq g-1,\quad 0\leq r<p,
\]
upstairs and $(a_1,b_1,\ldots,a_g,b_g)$ downstairs such that
\begin{equation}\label{eq:Qexplicit}
Q_{p,g}(e_0)=pa_1,\qquad Q_{p,g}(f_0)=b_1,
\end{equation}
while
\[
Q_{p,g}(e_{i,r})=a_{i+1},\qquad
Q_{p,g}(f_{i,r})=b_{i+1}.
\]
\end{lemma}

\begin{proof}
Choose the downstairs basis so that the covering character takes $a_1$
to $1\in\mathbb Z/p$ and vanishes on all other basis elements.  Decompose
$\Sigma_g$ as the union of a once-punctured torus $T$, carrying
$a_1,b_1$, and a genus-$(g-1)$ surface $R$ with one boundary component,
carrying the remaining symplectic pairs.  The induced cover of $T$ is a
connected genus-one surface $\widetilde T$ with $p$ boundary components:
it has Euler characteristic $-p$, its boundary monodromy is trivial, and
the monodromy along $a_1$ is a $p$-cycle.  The inverse image of $R$ is the
disjoint union of $p$ copies $R_0,\ldots,R_{p-1}$.  The closed covering
surface is obtained by gluing the boundary of $R_r$ to the corresponding
boundary component of $\widetilde T$.

Choose a symplectic pair $e_0,f_0$ in the genus-one part of
$\widetilde T$, disjoint from its boundary, so that $e_0$ traverses the
$a_1$ direction through all $p$ sheets and $f_0$ is a lift of $b_1$.
Then
\[
\begin{aligned}
e_0\cdot f_0&=1,\\
Q_{p,g}(e_0)&=pa_1,\qquad Q_{p,g}(f_0)=b_1.
\end{aligned}
\]
In each $R_r$, lift the standard symplectic pairs on $R$ and denote them
by $(e_{i,r},f_{i,r})$, $1\leq i\leq g-1$.  Pairs in different $R_r$'s
are disjoint, all of them are disjoint from $e_0,f_0$, and their only
nonzero intersections are
$e_{i,r}\cdot f_{i,r}=1$.  Their push-forwards are
$a_{i+1}$ and $b_{i+1}$, respectively.

Thus the displayed curves have the standard symplectic intersection
matrix and give
\[
1+p(g-1)=G
\]
symplectic pairs, hence $2G$ homology classes.  Since a genus-$G$ surface
has $H_1$ of rank $2G$, the unimodular intersection matrix shows that
these classes form an integral basis.  This proves all the asserted
push-forward formulas.
\end{proof}

Thus, after grouping columns by symplectic pairs, $Q_{p,g}$ has one
distinguished block $\operatorname{diag}(p,1)$ and, for each $i\geq1$,
$p$ identical copies mapping onto the $(i+1)$st symplectic plane.
In particular the three model cases below are governed by
\[
Q_{4,2}:\mathbb Z^{10}\to\mathbb Z^4,\qquad
Q_{3,3}:\mathbb Z^{14}\to\mathbb Z^6,
\]
and
\[
Q_{6,2}:\mathbb Z^{14}\to\mathbb Z^4,\qquad
Q_{2,4}:\mathbb Z^{14}\to\mathbb Z^8.
\]
In this basis the index defect is explicit: the cokernel of $Q_{p,g}$ is $\mathbb Z/p$.

\subsubsection{The symmetric $(4,4)$ construction}
Let
\[
F_{2,4}\subset \Sigma_2\times S^2
\]
be the connected cyclic degree-four multisection.  It has genus five, square zero, and class $4[\Sigma_2\times\{\mathrm{pt}\}]$.  Put
\[
M_{4}=M_{2,4}=(\Sigma_2\times S^2)\setminus\operatorname{int}\nu F_{2,4}.
\]
The first higher symmetric sum is
\[
Z^{4,4}_{2,2}(\varphi)=M_4\cup_\varphi M_4.
\]
It satisfies
\[
e(Z^{4,4}_{2,2})=8,\qquad \sigma(Z^{4,4}_{2,2})=0.
\]
Moreover
\[
H_1(M_4;\mathbb Z)\cong\mathbb Z^4\oplus\mathbb Z/4,
\qquad
\pi_1(Z^{4,4}_{2,2})\cong
\pi_1(M_4)*_{\pi_1(\Sigma_5)\times\mathbb Z}\pi_1(M_4).
\]
For an equatorially product-framed gluing induced by $S\in\operatorname{Sp}(10,\mathbb Z)$,
\[
H_1(Z^{4,4}_{2,2};\mathbb Z)
\cong
\operatorname{coker}
\begin{pmatrix}Q_{4,2}\\-Q_{4,2}S\end{pmatrix}
\oplus\mathbb Z/4.
\]
Here the surface matrix has size $8\times10$.  Thus a rational-cohomology realization is obtained exactly when it has rank eight over $\mathbb Q$.  In that case $|H_1|$ is divisible by $4$.  After replacing the equatorial framings by suitable spin-compatible symplectic framings, Theorem~\ref{thm:spin} gives a possibly different product gluing which is spin and has free intersection lattice $3H$.  Relative to the original equatorial framings, that gluing may have a nonzero framing term $\lambda$.

The cyclic cover associated to the fixed choices of Proposition~\ref{prop:cover} has degree four.  Its restriction to either complement is the product
\[
\Sigma_5\times P_4.
\]
Thus the lifted decomposition consists of two product pieces glued along copies of $\Sigma_5\times S^1$.  The $(4,4)$ case provides an even-degree model for the higher cyclic cover.

\subsubsection{The symmetric $(3,3)$ construction}
Now take the degree-three multisection
\[
F_{3,3}\subset\Sigma_3\times S^2.
\]
Its genus is seven.  With
\[
M_3=M_{3,3},
\]
we form
\[
Z^{3,3}_{3,3}(\varphi)=M_3\cup_\varphi M_3.
\]
The complement has
\[
H_1(M_3;\mathbb Z)\cong\mathbb Z^6\oplus\mathbb Z/3,
\]
and
\[
\pi_1(Z^{3,3}_{3,3})\cong
\pi_1(M_3)*_{\pi_1(\Sigma_7)\times\mathbb Z}\pi_1(M_3).
\]
For an equatorially product-framed gluing induced by $S\in\operatorname{Sp}(14,\mathbb Z)$,
\[
H_1(Z^{3,3}_{3,3};\mathbb Z)
\cong
\operatorname{coker}
\begin{pmatrix}Q_{3,3}\\-Q_{3,3}S\end{pmatrix}
\oplus\mathbb Z/3.
\]
The surface matrix is $12\times14$.  Full row rank gives the rational cohomology of $\#_3(S^2\times S^2)$, and every such equatorially product-framed realization has first homology of order divisible by three.

Since the covering degree is odd, pullback preserves the Arf invariant of a spin structure on the base.  One may begin with spin structures of the same Arf invariant on the two copies of $\Sigma_3$, pull them back to $\Sigma_7$, and then choose spin-compatible symplectic framings and a gluing in the corresponding level-two congruence coset while retaining full rational rank.  This may be a different realization from the equatorially product-framed one in the preceding paragraph.  The resulting manifold is spin with free lattice $3H$.

Here
\[
n=6,\qquad \ell=\operatorname{lcm}(3,3)=3.
\]
Thus the associated cyclic cover has degree three, not six.  Over either side it is a single product piece
\[
\Sigma_7\times P_3.
\]
Here the integer controlling the genus differs from the degree of the associated cyclic cover.  Identifying the closed cover after the boundary gluing remains open.

\subsubsection{The asymmetric $(6,2)$ construction}
The third model uses
\[
F_{2,6}\subset\Sigma_2\times S^2,
\qquad
F_{4,2}\subset\Sigma_4\times S^2.
\]
Both surfaces have genus seven, since
\[
6(2-1)=2(4-1)=6.
\]
Set
\[
M_6=M_{2,6},\qquad M_2=M_{4,2}.
\]
Then
\[
Z^{6,2}_{2,4}(\varphi)=M_6\cup_\varphi M_2
\]
has $e=8$ and signature zero, while
\[
H_1(M_6;\mathbb Z)\cong\mathbb Z^4\oplus\mathbb Z/6,
\qquad
H_1(M_2;\mathbb Z)\cong\mathbb Z^8\oplus\mathbb Z/2.
\]
The fundamental group is
\[
\pi_1(Z^{6,2}_{2,4})\cong
\pi_1(M_6)*_{\pi_1(\Sigma_7)\times\mathbb Z}\pi_1(M_2).
\]

For an equatorially product-framed gluing,
\[
H_1(Z^{6,2}_{2,4};\mathbb Z)
\cong
\operatorname{coker}
\begin{pmatrix}Q_{6,2}\\-Q_{2,4}S\end{pmatrix}
\oplus\mathbb Z/2,
\]
where the surface matrix again has size $12\times14$.  Hence the $(6,2)$ and $(3,3)$ constructions have identical rational matrix dimensions but different integral covering defects.  In the present case the order of the first homology group of every rational-cohomology equatorially product-framed realization is divisible by six.

The cyclic cover associated to the fixed choices of Proposition~\ref{prop:cover} has degree six.  Its inverse image over $M_6$ is one copy of
\[
\Sigma_7\times P_6,
\]
whereas its inverse image over $M_2$ consists of three copies of
\[
\Sigma_7\times P_2.
\]
The lifted decomposition is asymmetric in the number of components, unlike the symmetric $(3,3)$ construction with the same genus-seven boundary.

\begin{proposition}\label{prop:k3comparison}
The three $k=3$ model constructions above have the same Euler characteristic, signature, and rational Betti numbers and admit spin realizations with free lattice $3H$, but their mixed-sum presentations have different complement homology and associated cyclic covers.  In particular, the common gluing genus does not determine either the integral complement data or the cyclic-cover decomposition.  No claim is made here that arbitrary closed realizations of the three types are pairwise nondiffeomorphic.
\end{proposition}

\begin{proof}
The numerical invariants and spin statement were established above.  The complement homology groups are respectively
\[
\mathbb Z^4\oplus\mathbb Z/4,
\qquad
\mathbb Z^6\oplus\mathbb Z/3,
\qquad
(\mathbb Z^4\oplus\mathbb Z/6,\ \mathbb Z^8\oplus\mathbb Z/2),
\]
and the cyclic-cover degrees are $4,3,6$.  For the two genus-seven cases, the $(3,3)$ cover has one $\Sigma_7\times P_3$ piece over each side, whereas the $(6,2)$ cover has one $\Sigma_7\times P_6$ piece on one side and three $\Sigma_7\times P_2$ pieces on the other.  This proves the assertion.
\end{proof}

\begin{corollary}\label{cor:k3lattice}
For each of the three model constructions there is a minimal spin symplectic gluing $Z$ with
\[
H^*(Z;\mathbb Q)\cong H^*(\#_3(S^2\times S^2);\mathbb Q),\qquad
Q_{Z,\mathrm{free}}\cong 3H.
\]
Its fundamental group contains $\pi_1(\Sigma_G)\times\mathbb Z$.  Consequently none of these direct sums is homotopy equivalent to a closed real-hyperbolic four-manifold or to a compact quotient of the bidisk by a torsion-free cocompact irreducible lattice.
\end{corollary}

\begin{proof}
Choose the gluing by Theorem~\ref{thm:spin}.  Minimality follows from Theorem~\ref{thm:min}, and the subgroup statement follows from Corollary~\ref{cor:pi1}.  The two lattice exclusions follow from Theorem~\ref{thm:lattice}.
\end{proof}

\subsection{The remaining two \texorpdfstring{$k=3$}{k=3} types}
The two remaining types are
\[
(4,2;3,5;9),\qquad (3,2;5,7;13).
\]
Their complement homology is
\[
\mathbb Z^6\oplus\mathbb Z/4\quad\hbox{and}\quad
\mathbb Z^{10}\oplus\mathbb Z/2
\]
in the first case, and
\[
\mathbb Z^{10}\oplus\mathbb Z/3\quad\hbox{and}\quad
\mathbb Z^{14}\oplus\mathbb Z/2
\]
in the second.  Their associated cyclic-cover degrees are four and six.  The rational Mayer--Vietoris maps have sizes $16\times18$ and $24\times26$, respectively.  The two-dimensional excess persists throughout the $k=3$ list, while the integral matrices grow with the gluing genus.

\subsection{The seven types for \texorpdfstring{$5(S^2\times S^2)$}{5(S2 x S2)}}
Here $t=3$. The complete list is
\[
\begin{array}{c|c|c|c}
(p,q)&n&(g,h)&G\\ \hline
(5,5)&5&(2,2)&6\\ (6,3)&6&(2,3)&7\\ (8,2)&8&(2,5)&9\\
(3,3)&9&(4,4)&10\\ (5,2)&10&(3,6)&11\\ (4,2)&12&(4,7)&13\\ (3,2)&18&(7,10)&19.
\end{array}
\]
For each row, $e=12$, $\sigma=0$, and $G=g+h+2$.  A maximal-rank rational surface map has four-dimensional kernel.  Thus suitable gluings have
\[
b_1=0,\quad b_2=10,\quad b_2^+=b_2^-=5,
\]
and the rational cohomology of $\#_5(S^2\times S^2)$.  Spin gluings have free lattice $5H$.

There are two symmetric types.  The $(5,5)$ case uses two copies of $\Sigma_2\times S^2$ and a genus-six multisection; the $(3,3)$ case uses two copies of $\Sigma_4\times S^2$ and a genus-ten multisection.  The other five types are asymmetric.  Four have a degree-two side, and the gluing genus reaches nineteen for $(3,2)$.

The integral complement data already distinguish these constructions.  For example,
\[
H_1(M_{2,5};\Z)\cong\Z^4\oplus\Z/5,
\]
whereas the $(8,2)$ type uses
\[
H_1(M_{2,8};\Z)\cong\Z^4\oplus\Z/8,\qquad H_1(M_{5,2};\Z)\cong\Z^{10}\oplus\Z/2.
\]
The cyclic-cover degrees are
\[
5,\ 6,\ 8,\ 3,\ 10,\ 4,\ 6.
\]
Thus neither $G$ nor the cyclic-cover degree alone determines the higher type.

\subsection{What changes from \texorpdfstring{$k=1$}{k=1}}
For $k=3$ and $k=5$ the surface map is rectangular with excess dimensions two and four.  The rational realization argument remains uniform, but the integral problem is no longer controlled by one determinant: it involves the determinantal divisors of all maximal minors.  The spin theorem and boundary incompressibility apply uniformly.  Thus all twelve types admit spin rational-cohomology realizations with free form $kH$, and none of the direct sums is simply connected.

\section{Integral and complex-geometric consequences}

We next consider integral first homology and possible complex structures.

\subsection{Integral first homology}
Let $Z$ be an equatorially product-framed higher mixed sum with $H_1(Z;\Q)=0$.  The integral Mayer--Vietoris presentation of Section~5 shows that $H_1(Z;\Z)$ is finite.  The two covering defects survive as quotients $\Z/p$ and $\Z/q$.  Therefore
\[
\operatorname{lcm}(p,q)\mid |H_1(Z;\Z)|.
\]
For the five $k=3$ types the resulting lower divisors are
\[
4,\ 3,\ 6,\ 4,\ 6,
\]
and in the seven $k=5$ types they are
\[
5,\ 6,\ 8,\ 3,\ 10,\ 4,\ 6.
\]
Stronger product divisibility may hold for particular integral gluings, but it requires information about the maximal minors of the rectangular surface matrix and is not asserted here without that calculation.

For a rectangular integral presentation matrix of full row rank, the order of its cokernel is the greatest common divisor of its maximal minors.  The higher integral problem is therefore a determinantal-divisor problem, replacing the single determinant that appears when $k=1$.  A one-parameter family of higher gluings has unbounded first homology if the greatest common divisor of its maximal minors is unbounded.

\subsection{Complex structures}

\begin{proposition}\label{prop:complexhigher}
Let $Z$ be one of the rational-cohomology higher mixed sums constructed above, with $k>1$, and suppose that $Z$ admits an orientation-compatible complex structure.  Then $Z$ is a minimal surface of general type and
\[
 q(Z)=0,\qquad p_g(Z)=\frac{k-1}{2},\qquad
 c_2(Z)=2(k+1),\qquad K_Z^2=4(k+1),
\]
so that
\[
 K_Z^2=2c_2(Z),\qquad \chi(\mathcal O_Z)=\frac{k+1}{2}.
\]
If $Z$ is a direct mixed sum, its fundamental group is not a torsion-free cocompact irreducible lattice in $\operatorname{PSL}_2(\R)\times\operatorname{PSL}_2(\R)$.
\end{proposition}

\begin{proof}
The rational cohomology assumption gives $b_1(Z)=0$, $e(Z)=2+2k$, and $\sigma(Z)=0$.  A compact complex surface with even first Betti number is K\"ahler \cite{BHPV}.  The manifolds constructed above are symplectically minimal by Theorem~\ref{thm:min}.  Since $k>1$ gives $b_2^+=k>1$, symplectic minimality is equivalent here to smooth minimality \cite{HamiltonKotschick}.  Therefore the complex surface is minimal as well.

It cannot be rational: a minimal rational surface has $b_2=1$ or $2$, whereas here $b_2=2k>2$.  It cannot be ruled over a positive-genus curve because $b_1=0$.  A minimal K\"ahler surface of Kodaira dimension zero or one has $K^2=0$, whereas
\[
K_Z^2=2e(Z)+3\sigma(Z)=4(k+1)>0.
\]
Thus $Z$ is of general type \cite{BHPV}.  Hodge decomposition gives $q=b_1/2=0$.  Noether's formula gives
\[
\chi(\mathcal O_Z)=\frac{K_Z^2+c_2(Z)}{12}
=\frac{4(k+1)+2(k+1)}{12}=\frac{k+1}{2}.
\]
Since $\chi(\mathcal O_Z)=1-q+p_g$, we obtain $p_g=(k-1)/2$.  The last assertion follows from Theorem~\ref{thm:lattice}.
\end{proof}

For $k=3$, any such complex surface would therefore have
\[
(p_g,q,c_2,K^2)=(1,0,8,16),
\]
and for $k=5$ it would have
\[
(p_g,q,c_2,K^2)=(2,0,12,24).
\]
The equality $K^2=2c_2$ is the numerical equality satisfied by bidisk quotients.  Our group-theoretic argument excludes the irreducible bidisk-lattice case for the direct sums.  It also excludes compact complex-hyperbolic uniformization, including cocompact Deligne--Mostow lattices, because the fundamental group contains $\mathbb Z^2$.  It does not exclude reducible bidisk quotients or complex surfaces with a different universal cover.
There is also an independent numerical obstruction to ball uniformization:
a compact complex two-ball quotient satisfies $c_1^2=3c_2$, whereas the
surfaces considered here would satisfy $c_1^2=2c_2$.

\subsection{Comparison with mixed surfaces isogenous to a product}
The word ``mixed'' has a different meaning in the theory of surfaces
isogenous to a product.  In Catanese's terminology \cite{Catanese}, a
mixed action is a finite group action on $C\times C$ with elements
interchanging the two factors, and the resulting complex surface is a
free quotient of the product.  The mixed sums in this paper are instead
symplectic gluings of complements of multisections.  No finite quotient
presentation or mixed action is built into the construction.  If a
direct mixed sum were diffeomorphic to a surface isogenous to a product,
its fundamental group would have to contain a finite-index product of
surface groups, and its numerical invariants would have to agree with
the quotient formulas.  The Bass--Serre decomposition obtained here
constrains, but does not establish or exclude, such a virtual product
structure.  No identification of the two constructions is asserted.

\section{What remains open}

Theorems~\ref{thm:classification} and~\ref{thm:param} settle the numerical geography of the direct cyclic construction: all types occur at odd $k$, and none occurs at even $k$.  Every admissible type has both a minimal symplectic realization with the required rational cohomology and a spin realization with free intersection lattice $kH$.  Changing the boundary gluing alone cannot produce a manifold with the integral homology of $\#_k(S^2\times S^2)$ or a simply connected manifold.  Proposition~\ref{prop:MV} leaves the covering defects visible in first homology, while boundary incompressibility gives
\[
\pi_1(\Sigma_G)\times\mathbb Z\hookrightarrow\pi_1(Z).
\]
Proposition~\ref{prop:complexhigher} also determines the numerical invariants of any orientation-compatible complex structure and excludes irreducible bidisk uniformization for the direct sums.

Three questions remain.

\begin{question}\label{q:smithnormal}
For the five $k=3$ types, and already for the three model cases $(4,4)$, $(3,3)$, and $(6,2)$ studied in detail above, which finite abelian groups occur as $H_1(Z;\mathbb Z)$?  In particular, what are the minimal possible orders in these three cases?  The quotient argument gives the lower divisors
\[
4,\qquad 3,\qquad 6,
\]
for $(4,4)$, $(3,3)$, and $(6,2)$, respectively; determining whether stronger divisibility is forced is part of the Smith-form problem.
\end{question}

For an equatorially product-framed gluing, Question~\ref{q:smithnormal} is the Smith problem for
\[
A_S=\begin{pmatrix}Q_{p,g}\\-Q_{q,h}S\end{pmatrix},
\qquad S\in\operatorname{Sp}(2G,\mathbb Z).
\]
When $A_S$ has full row rank, the order of its cokernel is the greatest common divisor of its maximal minors.  The explicit covering bases in Section~11 therefore reduce Question~\ref{q:smithnormal} to a concrete family of determinantal-divisor calculations.  This is a genuine difference from $k=1$, where the corresponding matrix is square and a single determinant controls the finite cokernel.  The present paper proves the universal divisibility $\operatorname{lcm}(p,q)\mid |H_1(Z;\mathbb Z)|$; stronger divisibility and sharpness depend on the maximal determinantal divisors.

\begin{question}\label{q:smoothfamilies}
For a fixed odd $k>1$ and a fixed numerical type, are there infinitely many pairwise nondiffeomorphic direct mixed sums with the rational cohomology of $\#_k(S^2\times S^2)$?
\end{question}

A positive answer would follow from a one-parameter family $S_m\in\operatorname{Sp}(2G,\mathbb Z)$ for which the maximal determinantal divisor of $A_{S_m}$ is unbounded.  The resulting first homology groups would have unbounded order, giving pairwise nondiffeomorphic manifolds after passage to a subsequence.  The transvection argument used when $k=1$ still gives one-parameter symplectic families, but in the rectangular case it does not control the common divisor of all maximal minors.  The existence of such families remains an integral problem.

\begin{question}\label{q:complex}
Do any of the $k=3$ or $k=5$ direct mixed sums admit orientation-compatible complex structures?
\end{question}

At the $k=3$ level such a surface would have
\[
(p_g,q,c_2,K^2)=(1,0,8,16),
\]
and at the $k=5$ level it would have
\[
(p_g,q,c_2,K^2)=(2,0,12,24).
\]
In both cases $K^2=2c_2$.  The subgroup $\pi_1(\Sigma_G)\times\mathbb Z$ rules out a torsion-free cocompact irreducible bidisk lattice, while its $\mathbb Z^2$ subgroup rules out compact complex-hyperbolic uniformization, including cocompact Deligne--Mostow lattices.  Reducible bidisk quotients are not excluded by these arguments.  Determining whether one of the mixed sums is homotopy equivalent, or diffeomorphic, to a surface isogenous to a product appears to require substantially finer information about its fundamental group than its abelianization.

A related direction is to modify the complements before gluing.  The direct construction cannot become simply connected, but Luttinger surgeries on Lagrangian tori in the complements may change both the boundary image in first homology and the nonabelian relations without changing Euler characteristic or signature.  Such a construction is not a direct mixed sum in the sense of this paper and belongs to a separate surgery problem.

\appendix
\section{Complete numerical list through \texorpdfstring{$k=15$}{k=15}}\label{app:table}

The following table gives the complete list of unordered numerical types through $k=15$.  It records the output of Theorem~\ref{thm:param} and answers the finite-range classification question posed in \cite{AkhmedovBraided}.  Each row records $(p,q;n;g,h;G)$.

\begin{center}
\footnotesize
\setlength{\LTleft}{0pt}
\setlength{\LTright}{0pt}
\begingroup
\hbadness=10000
\begin{longtable}{@{}c c >{\raggedright\arraybackslash}p{0.79\textwidth}@{}}
\toprule
$k$ & number & complete unordered types $(p,q;n;g,h;G)$\\
\midrule
\endfirsthead
\toprule
$k$ & number & complete unordered types $(p,q;n;g,h;G)$\\
\midrule
\endhead
1&3&
$(3,3;3;2,2;4)$,\quad $(4,2;4;2,3;5)$,\quad
$(3,2;6;3,4;7)$\\
\midrule
3&5&
$(4,4;4;2,2;5)$,\quad $(3,3;6;3,3;7)$,\quad
$(6,2;6;2,4;7)$,\newline
$(4,2;8;3,5;9)$,\quad $(3,2;12;5,7;13)$\\
\midrule
5&7&
$(5,5;5;2,2;6)$,\quad $(3,3;9;4,4;10)$,\quad
$(6,3;6;2,3;7)$,\newline
$(8,2;8;2,5;9)$,\quad $(5,2;10;3,6;11)$,\quad
$(4,2;12;4,7;13)$,\newline
$(3,2;18;7,10;19)$\\
\midrule
7&7&
$(6,6;6;2,2;7)$,\quad $(4,4;8;3,3;9)$,\quad
$(3,3;12;5,5;13)$,\newline
$(10,2;10;2,6;11)$,\quad $(6,2;12;3,7;13)$,\quad
$(4,2;16;5,9;17)$,\newline
$(3,2;24;9,13;25)$\\
\midrule
9&9&
$(7,7;7;2,2;8)$,\quad $(3,3;15;6,6;16)$,\quad
$(8,4;8;2,3;9)$,\newline
$(4,3;12;4,5;13)$,\quad $(9,3;9;2,4;10)$,\quad
$(12,2;12;2,7;13)$,\newline
$(7,2;14;3,8;15)$,\quad $(4,2;20;6,11;21)$,\quad
$(3,2;30;11,16;31)$\\
\midrule
11&11&
$(8,8;8;2,2;9)$,\quad $(5,5;10;3,3;11)$,\quad
$(4,4;12;4,4;13)$,\newline
$(3,3;18;7,7;19)$,\quad $(6,3;12;3,5;13)$,\quad
$(14,2;14;2,8;15)$,\newline
$(8,2;16;3,9;17)$,\quad $(6,2;18;4,10;19)$,\quad
$(5,2;20;5,11;21)$,\newline
$(4,2;24;7,13;25)$,\quad $(3,2;36;13,19;37)$\\
\midrule
13&10&
$(9,9;9;2,2;10)$,\quad $(3,3;21;8,8;22)$,\quad
$(10,5;10;2,3;11)$,\newline
$(6,4;12;3,4;13)$,\quad $(5,3;15;4,6;16)$,\quad
$(12,3;12;2,5;13)$,\newline
$(16,2;16;2,9;17)$,\quad $(9,2;18;3,10;19)$,\quad
$(4,2;28;8,15;29)$,\newline
$(3,2;42;15,22;43)$\\
\midrule
15&10&
$(10,10;10;2,2;11)$,\quad $(6,6;12;3,3;13)$,\quad
$(4,4;16;5,5;17)$,\newline
$(3,3;24;9,9;25)$,\quad $(12,4;12;2,4;13)$,\quad
$(18,2;18;2,10;19)$,\newline
$(10,2;20;3,11;21)$,\quad $(6,2;24;5,13;25)$,\quad
$(4,2;32;9,17;33)$,\newline
$(3,2;48;17,25;49)$\\
\bottomrule
\end{longtable}
\endgroup
\end{center}

\enlargethispage{3.5\baselineskip}
\section*{Acknowledgments}
We are grateful to Sai-Kee Yeung for his interest in this work and for very helpful discussions.  This paper continues the mixed-sum construction introduced in \cite{AkhmedovBraided} and uses the spin and cyclic-cover viewpoint developed further in \cite{AkhYeung}.

An LLM-based tool was used for limited assistance with computational preparation and grammar checking.  The authors checked the mathematical arguments and take full responsibility for the contents of the paper.

\end{document}